\documentclass[11pt]{article}

\usepackage[T1]{fontenc}
\usepackage[margin=1in]{geometry}
\usepackage{lmodern}
\usepackage{microtype}
\usepackage{amsmath,amssymb,amsthm,mathtools}
\usepackage{mathrsfs}
\usepackage{bm}
\usepackage{enumitem}
\usepackage[colorlinks=true,linkcolor=blue,citecolor=blue,urlcolor=blue]{hyperref}

\numberwithin{equation}{section}

\newcommand{\R}{\mathbb{R}}
\newcommand{\N}{\mathbb{N}}
\newcommand{\eps}{\varepsilon}
\newcommand{\Hs}{\mathcal{H}_{p,\mu}}
\newcommand{\E}{\mathcal{E}}
\newcommand{\Ls}{\mathcal{L}_{s}}
\renewcommand{\Cap}{\operatorname{Cap}}
\newcommand{\dist}{\operatorname{dist}}
\newcommand{\supp}{\operatorname{supp}}
\newcommand{\loc}{\mathrm{loc}}

\newcommand{\dd}{\,d}
\newcommand{\one}{\mathbf{1}}

\newtheorem{theorem}{Theorem}[section]
\newtheorem{lemma}{Lemma}[section]
\newtheorem{proposition}{Proposition}[section]
\newtheorem{corollary}{Corollary}[section]

\theoremstyle{definition}
\newtheorem{definition}{Definition}[section]

\theoremstyle{remark}

\title{Symmetry and Monotonicity for Subcritical Fractional Hartree Equations with Singular Sets of Zero Capacity}

\author{Guangze Gu, Aleks Jevnikar, Zhipeng Yang\thanks{Corresponding author: yangzhipeng326@163.com}}

\date{}

\AtEndDocument{%
\par
\bigskip
\bigskip

\noindent
\textbf{Guangze Gu}\\[0.2em]
\textsc{Department of Mathematics, Yunnan Normal University, Kunming 650500, China}\\
\textsc{Yunnan Key Laboratory of Modern Analytical Mathematics and Applications, Kunming 650500, China}\\[0.3em]
\textit{E-mail address}: \texttt{guangzegu@163.com}

\bigskip

\noindent
\textbf{Aleks Jevnikar}\\[0.2em]
\textsc{Department of Mathematics, Computer Science and Physics, University of Udine, Udine 33100, Italy}\\[0.3em]
\textit{E-mail address}: \texttt{aleks.jevnikar@uniud.it}

\bigskip

\noindent
\textbf{Zhipeng Yang}\\[0.2em]
\textsc{Department of Mathematics, Yunnan Normal University, Kunming 650500, China}\\
\textsc{Yunnan Key Laboratory of Modern Analytical Mathematics and Applications, Kunming 650500, China}\\[0.3em]
\textit{E-mail address}: \texttt{yangzhipeng326@163.com}
}

\begin{document}

\maketitle

\begin{abstract}
We study positive weak singular solutions of the subcritical fractional Hartree equation
\[
(-\Delta)^s u = \left( \int_{\mathbb{R}^N\setminus\Gamma} \frac{u^p(y)}{|x-y|^\mu}d y \right)u^q \quad\text{in }\mathbb{R}^N\setminus\Gamma,
\]
where \(0<s<1\), \(N>2s\), \(0<\mu<N\), and \(\Gamma\subset\{x_1=0\}\) is compact with zero \((s,2)\)-capacity. For solutions with a nonremovable singular point on $\Gamma$, we prove hyperplane symmetry and strict transverse monotonicity under global Lebesgue integrability conditions. For $N\ge2$, we also prove these conclusions under upper bounds on $p$ and $q$, using only the integrability required for weak solutions. We also establish local regularity and construct solutions with logarithmic singularities along a line segment.
\end{abstract}



\section{Introduction and Main Results}\label{Sec1}

Hartree and Choquard equations couple elliptic diffusion with a Riesz potential in the nonlinearity. A standard local model is
\[
-\Delta u+u=(I_\alpha*|u|^p)|u|^{p-2}u \quad\text{in }\R^N.
\]
We use the standard convention \(I_\alpha(x)=c_{N,\alpha}|x|^{\alpha-N}\), and write \(I_\alpha f=I_\alpha*f\) for the corresponding potential. We also set
\[
K_\rho(x)=|x|^{-\rho},\qquad 0<\rho<N,
\]
so that \(K_\mu=c_{N,N-\mu}^{-1}I_{N-\mu}\). Lieb \cite{Lieb1977} proved that the original Choquard equation has a positive minimizing solution, unique up to translations. Lions \cite{Lions1980} studied the equation by variational methods, and Ma and Zhao \cite{MaZhao} classified positive solitary solutions by symmetry methods. Moroz and Van Schaftingen \cite{MorozVanSchaftingen} subsequently established existence, positivity, regularity, and sharp decay of nonlinear Choquard ground states; their survey \cite{MorozVanSchaftingenGuide} reviews these developments.

The Hardy-Littlewood-Sobolev inequality determines the critical scaling of
\[
\iint_{\R^N\times\R^N} \frac{|u(x)|^r|u(y)|^r}{|x-y|^\mu}\,\dd x\dd y
\]
and, for fractional diffusion of order \(s\), gives the critical exponent \(2_{\mu,s}^*=(2N-\mu)/(N-2s)\). Moroz and Van Schaftingen \cite{MorozVanSchaftingenCritical} studied ground states at the lower HLS-critical threshold, while Gao and Yang \cite{GaoYang} treated critical Dirichlet Choquard equations. In the fractional setting, d'Avenia, Siciliano, and Squassina \cite{DAveniaSicilianoSquassina} established regularity, existence, nonexistence, symmetry, and decay results, and Mukherjee and Sreenadh \cite{MukherjeeSreenadh} studied a Brezis-Nirenberg type problem at the exponent~\(2_{\mu,s}^*\).

For singular solutions, the basic questions are removability, classification of blow-up profiles, and symmetry near the singular set. In the local semilinear theory, one studies
\[
-\Delta u=f(u)\quad\text{in }\Omega\setminus K.
\]
Brezis and V\'eron \cite{BrezisVeron} obtained removability results for sufficiently strong absorption, Caffarelli, Gidas, and Spruck \cite{CaffarelliGidasSpruck} proved asymptotic radial symmetry and identified the behavior of critical solutions near isolated singularities, and V\'azquez and V\'eron \cite{VazquezVeron} further developed the classification theory. Gidas, Ni, and Nirenberg \cite{GidasNiNirenberg} applied moving planes to semilinear elliptic equations, while Berestycki and Nirenberg \cite{BerestyckiNirenberg} developed moving-plane and sliding arguments for further classes of equations.

Fractional diffusion requires comparison principles that account for interactions across the reflecting plane. The extension of Caffarelli and Silvestre \cite{CaffarelliSilvestre} provides one approach, while direct methods work at the operator or integral level. Chen, Li, and Ou \cite{ChenLiOu} developed direct moving planes for conformally invariant integral equations of the form
\[
u(x)=\int_{\R^N}\frac{u(y)^p}{|x-y|^{N-\alpha}}\,\dd y.
\]
Li \cite{Li2004} used moving spheres to study related integral equations. For the fractional Laplacian, Chen, Li, and Li \cite{ChenLiLi} established antisymmetric maximum and narrow-region principles, while Chen, Li, and Zhang \cite{ChenLiZhang} developed a direct moving-spheres method. In fractional Hartree problems, Dai, Fang, and Qin \cite{DaiFangQin} used direct moving planes to classify critical positive solutions, and Dai, Huang, Qin, Wang, and Fang \cite{DaiHuangQinWangFang} obtained further regularity and classification results for static equations. Dai, Liu, and Qin \cite{DaiLiuQin2021} established classification and Liouville results for nonnegative entire classical solutions of fractional Hartree equations with independent powers \(p\) and \(q\) under the separate Kelvin bounds.

Isolated singularities of Hartree and related elliptic equations have also been studied. Andrade, Feng, Piccione, and Yang \cite{AndradeFengPiccioneYang2025} studied the critical local Hartree equation
\[
-\Delta u=(R_\alpha*F(u))f(u) \quad\text{in }B_R^*\text{ or in }\R^n\setminus\{0\},
\]
where \(B_R^*:=B_R\setminus\{0\}\), \(R_\alpha(x)=|x|^{\alpha-n}\), \(f(\xi)=|\xi|^{p-2}\xi\), \(F'=f\), and \(p=2_\alpha^*={(n+\alpha)/(n-2)}\). Under their regularity and integrability assumptions, they proved radial symmetry of entire positive singular solutions and derived sharp local asymptotics through blow-up analysis, a local integral representation, and asymptotic moving spheres. For the critical fractional Hartree equation, the same authors \cite{AndradeFengPiccioneYang2026Delaunay} proved radial symmetry of positive classical solutions with a nonremovable isolated singularity. Under an additional restriction on the Riesz exponent, they also constructed Delaunay solutions with nonconstant periodic Emden--Fowler profiles for all sufficiently large periods. In the planar exponential regime, Feng, Yang, and Zhou \cite{FengYangZhou2026} considered
\[
-\Delta u=(K_\alpha*e^u)e^u \quad\text{in }B_1\setminus\{0\}\subset\R^2,
\]
and obtained the asymptotic behavior of finite-total-curvature singular solutions by means of a representation formula. From an operator-theoretic viewpoint, Boni, Noja, and Scandone \cite{BoniNojaScandone2026} related the punctured semilinear equation to a stationary nonlinear Schr\"odinger equation with a point interaction:
\[
\begin{aligned}
(-\Delta+\lambda)u&=\sigma|u|^{p-1}u &&\text{in }\R^d\setminus\{0\},\quad d=2,3,\\
(-\Delta_\vartheta+\lambda)u&=\sigma|u|^{p-1}u\qquad&&\text{in }\R^d.
\end{aligned}
\]
Their equivalence yields an operator-theoretic and variational description of isolated singularities, together with multiplicity and nodal consequences in the focusing case.

For singular sets of zero capacity, Montoro, Punzo, and Sciunzi \cite{MontoroPunzoSciunzi} considered positive weak solutions of
\[
(-\Delta)^s u=f(x,u)\quad\text{in }\Omega\setminus\Gamma.
\]
They imposed zero exterior data, with \(\Omega\) either a bounded domain or the whole space and \(\Gamma\) compact with zero fractional capacity, and proved symmetry and monotonicity. Cai, Gu, and Jevnikar \cite{CaiGuJevnikar2025} treated the local Hartree problem
\[
-\Delta u=(K_\mu*F(u))f(u) \quad\text{in }\R^N\setminus K, \qquad K\subset\{x_1=0\},
\]
where \(K\) is compact of zero capacity, and proved symmetry and monotonicity of singular solutions with respect to the singular hyperplane. Gu, Jevnikar, and Yang \cite{GuJevnikarYang2026} studied the punctured fractional equation
\[
(-\Delta)^s u= \left(\int_{\R^N\setminus\{0\}} \frac{u^p(y)}{|x-y|^\mu}\,\dd y\right)u^q \quad\text{in }\R^N\setminus\{0\},
\]
for \(0<s<1\), \(N>2s\), \(0<\mu<N\), and \(p,q>0\). Writing \(H_{p,\mu}[u](x)=\int_{\R^N\setminus\{0\}}u^p(y)|x-y|^{-\mu}\,\dd y\) and \(G=H_{p,\mu}[u]u^q\), they derived, under the bound \(u(x)=O(|x|^{-(N-2s)})\) as \(x\to0\) and a weighted integrability condition on \(G\), the Riesz decomposition
\[
u(x)=c_{N,s}\int_{\R^N} \frac{G(y)}{|x-y|^{N-2s}}\,\dd y +\kappa|x|^{-(N-2s)},\qquad \kappa\ge0.
\]
For \(p,q\ge1\), nonnegative Kelvin exponents, and \(u(x)\to+\infty\) as \(x\to0\), they retained the fundamental-solution term in a moving-spheres proof of radial symmetry and strict radial decrease. For the critical pair
\[
p=\frac{2N-\mu}{N-2s},\qquad q=\frac{N+2s-\mu}{N-2s},
\]
Gu, Jevnikar, and Yang \cite{GuJevnikarYangCritical} proved hyperplane symmetry and strict transverse monotonicity for positive continuous weak solutions with a nonremovable singular point on a compact set of zero \((s,2)\)-capacity contained in a hyperplane, for \(N\ge2\), \(0<s<1\), and the full range \(0<\mu<N\), without additional global Lebesgue integrability assumptions.

In this paper, we consider the equation
\begin{equation}\label{eq1.1}
(-\Delta)^s u = \Hs[u]u^q \quad\text{in }\R^N\setminus\Gamma,
\end{equation}
where
\begin{equation}\label{eq1.2}
\Hs[u](x)= \int_{\R^N\setminus\Gamma} \frac{u^p(y)}{|x-y|^\mu}\dd y,
\end{equation}
with
\[
0<s<1,\qquad N>2s,\qquad 0<\mu<N,\qquad p,q>1.
\]
For \(\lambda\in\R\), let
\[
T_\lambda=\{x_1=\lambda\}, \qquad R_\lambda x=x^\lambda=(2\lambda-x_1,x'), \qquad u_\lambda(x)=u(x^\lambda), \qquad \Gamma_\lambda=R_\lambda(\Gamma).
\]
For any function \(h\), write \(h_\lambda=h\circ R_\lambda\). In particular,
\[
V=\Hs[u], \qquad V_\lambda(x)=V(x^\lambda).
\]
For \(\lambda\le0\), set
\[
\Sigma_\lambda=\{x_1<\lambda\}, \qquad D_\lambda=\Sigma_\lambda\setminus(\Gamma\cup\Gamma_\lambda), \qquad w_\lambda=u-u_\lambda, \qquad W_\lambda=(w_\lambda)^+\one_{\Sigma_\lambda}.
\]
The geometry assumed below is that \(\Gamma\) is a compact set of zero \((s,2)\)-capacity contained in the fixed reflecting hyperplane. Since zero capacity implies zero Lebesgue measure, the potential in \eqref{eq1.2} may equivalently be computed by integration over all of \(\R^N\).

We write
\[
2_s^*=\frac{2N}{N-2s}, \qquad 2_{\mu,s}^*=\frac{2N-\mu}{N-2s}.
\]
The scaling of \eqref{eq1.1} described after Definition~\ref{Def1.4} preserves the homogeneous $H^s$ seminorm when
\[
p+q+1=2\,2_{\mu,s}^*.
\]
Throughout the paper we work in the strictly subcritical range
\[
p>1, \qquad q>1, \qquad p+q+1<2\,2_{\mu,s}^*.
\]
The HLS-admissibility bounds and the additional integrability used for regularity are specified in Definition~\ref{Def1.4} and analyzed immediately after it.

We use the Gagliardo seminorm
\[
[v]_{H^s(\R^N)}^2= \iint_{\R^N\times\R^N} \frac{|v(x)-v(y)|^2}{|x-y|^{N+2s}}\dd x\dd y, \qquad \|v\|_{H^s(\R^N)}^2=\|v\|_{L^2(\R^N)}^2+[v]_{H^s(\R^N)}^2.
\]
For an open set $U$, $H^s(U)$ is defined by the analogous integrals over $U$ and $U\times U$. We write $v\in H^s_{\loc}(\Omega)$ if $v|_U\in H^s(U)$ for every open set $U\Subset\Omega$.

\begin{definition}\label{Def1.1}
For a compact set \(K\subset\R^N\), define
\[
\Cap_{s,2}(K)= \inf\left\{ \|\varphi\|_{H^s(\R^N)}^2: \varphi\in C_c^\infty(\R^N),\ \varphi\ge 1\text{ in a neighborhood of }K \right\}.
\]
\end{definition}

\begin{definition}\label{Def1.2}
Set
\[
\Ls(\R^N)= \left\{v\in L^1_{\loc}(\R^N): \int_{\R^N}\frac{|v(x)|}{1+|x|^{N+2s}}\dd x<+\infty \right\}.
\]
A measurable function \(u\) is called a positive weak singular solution of \eqref{eq1.1} if
\[
u>0\quad\text{a.e. in }\R^N\setminus\Gamma, \qquad u\in\Ls(\R^N), \qquad u\in H^s_{\loc}(\R^N\setminus\Gamma),
\]
\[
\Hs[u](x)<+\infty\quad\text{for a.e. }x\in\R^N\setminus\Gamma, \qquad \Hs[u]u^q\in L^1_{\loc}(\R^N\setminus\Gamma),
\]
and
\[
\int_{\R^N}u(x)(-\Delta)^s\varphi(x)\dd x = \int_{\R^N\setminus\Gamma}\Hs[u]u^q\varphi\dd x
\]
for every \(\varphi\in C_c^\infty(\R^N\setminus\Gamma)\).
\end{definition}
A point is called regular for a function if it lies outside its singular set.

\begin{definition}\label{Def1.3}
A positive weak singular solution \(u\) has a nonremovable singular point \(x_0\in\Gamma\) if
\[
\lim_{r\downarrow0} \operatorname*{ess\,sup}_{B_r(x_0)\setminus\Gamma}u=+\infty.
\]
\end{definition}
This definition depends only on the almost-everywhere class of \(u\). Under the admissibility assumptions below, Proposition~\ref{Prop2.2} provides a unique locally H\"older continuous representative in \(\R^N\setminus\Gamma\), and the preceding condition is equivalent to the ordinary limsup of that representative at \(x_0\) being \(+\infty\).

We first introduce the global integrability assumptions used in Theorems~\ref{Thm1.1} and \ref{Thm1.2}. Theorem~\ref{Thm1.3} gives a second parameter range in which bounds on the nonlinear exponents replace the global integrability assumptions.

\begin{definition}\label{Def1.4}
Set
\begin{equation}\label{eq1.3}
m_H=\frac{N(p+q-1)}{N+2s-\mu}.
\end{equation}
We say that a measurable function \(v\) belongs to the admissible class
\[
\mathfrak A=\mathfrak A_{p,q,\mu,s}(\R^N)
\]
if
\begin{equation}\label{eq1.4}
m_H> \max\left\{ p,\, \frac{(p-1)2_s^*}{2_s^*-1},\, \frac{q\,2_s^*}{2_s^*-1},\, \frac{(q-1)2_s^*}{2_s^*-2} \right\},
\end{equation}
and if there exists an exponent \(m_B\in(1,+\infty)\) satisfying
\begin{equation}\label{eq1.5}
m_B>\max\left\{\frac{pN}{N-\mu},\,\frac{qN}{2s}\right\}
\end{equation}
such that
\begin{equation}\label{eq1.6}
v\in L^{m_H}(\R^N)\cap L^{m_B}(\R^N).
\end{equation}
\end{definition}

Whenever \(v\in\mathfrak A(\R^N)\), we fix once and for all one exponent \(m_B\) satisfying \eqref{eq1.5}-\eqref{eq1.6}. The quantities \(\beta_B,r_B,\alpha_B\), and \(\gamma_B\) below are defined using this exponent. The two exponents in Definition~\ref{Def1.4} have different roles. The exponent \(m_H\) satisfies the HLS exponent identity
\[
\frac{p+q-1}{m_H}+\frac{2}{2_s^*}+\frac{\mu}{N}=2,
\]
whereas $m_B$ gives boundedness of the potential and local regularity of $u$. The exponent $m_H$ also has a scaling interpretation: if $u$ solves \eqref{eq1.1}, then $u_R(x)=R^\alpha u(Rx)$ solves the equation with singular set $R^{-1}\Gamma$, where $\alpha=(N+2s-\mu)/(p+q-1)$. Since $m_H=N/\alpha$, the $L^{m_H}$ norm is invariant under this scaling.

Since \(N+2s-\mu>0\), direct algebra gives
\[
m_H<2_s^* \quad\Longleftrightarrow\quad (p+q-1)(N-2s)<2(N+2s-\mu) \quad\Longleftrightarrow\quad p+q+1<2\,2_{\mu,s}^*.
\]
The lower bounds in \eqref{eq1.4} are exactly the conditions needed for the four HLS exponents used in Proposition~\ref{Prop2.1} to lie in \((1,+\infty)\). Consequently the full subcritical HLS-admissible parameter condition is
\[
\max\left\{ p,\, \frac{2N(p-1)}{N+2s},\, \frac{2Nq}{N+2s},\, \frac{N(q-1)}{2s} \right\} < \frac{N(p+q-1)}{N+2s-\mu} < \frac{2N}{N-2s}.
\]
Equivalently, the left-hand inequality is the system
\[
\begin{aligned}
N(q-1)&>p(2s-\mu),\\
(N+2s)(p+q-1)&>2(p-1)(N+2s-\mu),\\
(N+2s)(p+q-1)&>2q(N+2s-\mu),\\
2s(p+q-1)&>(q-1)(N+2s-\mu).
\end{aligned}
\]
The inequalities in \eqref{eq1.4} are additional requirements beyond strict subcriticality. When the full condition holds, necessarily \(p<2_s^*\) and \(q<2_s^*-1\). There are singular solutions satisfying these conditions. For
\[
N=2,\qquad s=\frac12,\qquad\mu=1,\qquad p=q=2,
\]
one has $m_H=3$ and may take $m_B=5$. Proposition~\ref{Prop2.3} constructs positive solutions in $L^3(\R^2)\cap L^5(\R^2)$ with nonremovable singularities on the line segment $\Gamma=\{0\}\times[-1,1]$.

By Proposition~\ref{Prop2.2}, every solution satisfying the preceding admissibility assumptions has a unique locally H\"older continuous representative away from \(\Gamma\). Throughout the main results and the moving-plane argument, all pointwise statements about \(u\), including symmetry, monotonicity, and reflection identities, refer to this representative.

We may now state the principal results of the paper.

\begin{theorem}\label{Thm1.1}
Let \(0<s<1\), \(N>2s\), and \(0<\mu<N\). Let \(u\) be a positive weak singular solution of \eqref{eq1.1}, and assume that the following two conditions hold:
\begin{enumerate}[label=\textup{(H\arabic*)},ref=\textup{(H\arabic*)},leftmargin=4.5em]
\item\label{Ass1.1}
The singular set is such that
\[
\Gamma\subset\{x_1=0\}, \qquad \Gamma\text{ is compact}, \qquad \Cap_{s,2}(\Gamma)=0.
\]

\item\label{Ass1.2}
The parameters satisfy
\[
p>1, \qquad q>1, \qquad p+q+1<2\,2_{\mu,s}^*,
\]
and the solution belongs to the admissible integrability class
\(
u\in\mathfrak A(\R^N).
\)

\end{enumerate}
Suppose also that $u$ has at least one nonremovable singular point on $\Gamma$. Then
\begin{equation}\label{eq1.7}
u(x_1,x')=u(-x_1,x') \quad\text{for every }x\in\R^N\setminus\Gamma.
\end{equation}
Moreover,
\begin{equation}\label{eq1.8}
u(x_1,x')\le u(y_1,x') \quad\text{whenever }x_1<y_1<0.
\end{equation}
By symmetry, the corresponding reverse monotonicity holds on \(\{x_1>0\}\).
\end{theorem}

The monotonicity conclusion in Theorem~\ref{Thm1.1} admits the following refinement.

\begin{theorem}\label{Thm1.2}
Under the hypotheses of Theorem~\ref{Thm1.1}, for every \(x'\in\R^{N-1}\),
\[
u(x_1,x')<u(y_1,x') \qquad\text{whenever }x_1<y_1<0,
\]
and
\[
u(y_1,x')<u(x_1,x') \qquad\text{whenever }0<x_1<y_1.
\]
Equivalently, \(t\mapsto u(t,x')\) is strictly increasing on \((-\infty,0)\) and strictly decreasing on \((0,+\infty)\).
\end{theorem}

For $N\ge2$, a second result replaces condition \ref{Ass1.2} by separate upper bounds on $p$ and $q$ and uses only the integrability in Definition~\ref{Def1.2}.

\begin{theorem}\label{Thm1.3}
Let \(N\ge2\), \(0<s<1\), \(0<\mu<N\), and \(p,q>1\) satisfy
\begin{equation}\label{eq1.9}
p\le\frac{2N-\mu}{N-2s},\qquad q\le\frac{N+2s-\mu}{N-2s},
\end{equation}
with at least one strict inequality. Let \(\Gamma\) satisfy \ref{Ass1.1}, and let \(u\) be a positive weak singular solution of \eqref{eq1.1} in the sense of Definition~\ref{Def1.2}, with at least one nonremovable singular point on \(\Gamma\). Then \(u\) has a locally H\"older continuous representative in \(\R^N\setminus\Gamma\), and the symmetry and strict monotonicity conclusions of Theorems~\ref{Thm1.1} and \ref{Thm1.2} hold for this representative.
\end{theorem}

The bounds in \eqref{eq1.9}, with at least one strict inequality, imply strict subcriticality and, since $q>1$, also $\mu<4s$. Theorems~\ref{Thm1.1} and \ref{Thm1.3} are complementary. For example, $(N,s,\mu,p,q)=(3,\tfrac12,1,2,\tfrac85)$ satisfies $m_H=13/5<2_s^*=3$, and the maximum in \eqref{eq1.4} equals $12/5$, but $q=8/5$ exceeds the upper bound $3/2$ in \eqref{eq1.9}. The example below shows that Theorem~\ref{Thm1.3} also covers solutions outside $\mathfrak A$. The assumption $N\ge2$ allows Kelvin centers in $\{x_1=0\}\setminus\Gamma$ and a tangential direction along which these centers tend to infinity. Theorem~\ref{Thm1.1} also applies when $N=1>2s$.

To exhibit such a solution, take
\[
N=3,\qquad s=\frac45,\qquad\mu=\frac32,\qquad p=2,\qquad q=\frac95.
\]
The homogeneous solution of \cite[Theorem 1.3]{GuJevnikarYang2026} has the form
\[
u(x)=A|x|^{-31/28},
\]
where \(A\) is the unique positive constant specified in that theorem. With \(\alpha=31/28\), one has \(0<\alpha<N-2s\) and \(N-\mu<\alpha p<N\), which are the convergence and positivity conditions in that theorem. This solution satisfies the assumptions of Theorem~\ref{Thm1.3} with \(\Gamma=\{0\}\), whereas it belongs to no \(L^m(\R^N)\) with \(1\le m\le\infty\).

The proofs use moving planes and three ingredients. First, for \(\lambda<0\) and every \(x\in D_\lambda\), the reflected potentials satisfy
\[
V(x)-V_\lambda(x) = \int_{\Sigma_\lambda} \bigl(u^p(y)-u_\lambda^p(y)\bigr) \left( \frac1{|x-y|^\mu}-\frac1{|x^\lambda-y|^\mu} \right)\dd y,
\]
whose positive part \((V-V_\lambda)^+\) we estimate on the positivity set of the reflected difference. Second, when the reflected singular set enters the moving half-space, we derive energy estimates using compactly supported tests that vanish near the reflected singular set and then remove the cutoffs to establish finite fractional energy for \(W_\lambda\). Third, we pass from the weak formulation to pointwise comparison through local H\"older regularity, obtained from the stated integrability assumptions or directly from local Sobolev regularity under \eqref{eq1.9}.

The reflected energy estimate also gives a quantitative comparison criterion. For \(\lambda<0\), let
\[
E_\lambda=\{x\in D_\lambda:u(x)>u_\lambda(x)\}.
\]
Corollary~\ref{Cor4.2} proves that
\[
\|u\|_{L^{m_H}(E_\lambda)}<\tau_* \quad\Longrightarrow\quad W_\lambda\equiv0.
\]
The constant $\tau_*$ is independent of the reflecting plane.

Section~\ref{Sec2} contains the capacity, HLS, and regularity estimates, together with the construction of singular solutions. Section~\ref{Sec3} justifies tests near the reflected singular set. Section~\ref{Sec4} proves the reflected Hartree energy estimate, and Section~\ref{Sec5} proves Theorems~\ref{Thm1.1} and \ref{Thm1.2} by moving planes. Section~\ref{Sec6} proves Theorem~\ref{Thm1.3}.

\section{Preliminaries}\label{Sec2}

We collect the elementary analytic tools used in the moving-plane argument.  The capacity facts below are standard consequences of the definition of Bessel capacity; see, for instance, \cite{AdamsHedberg,DiNezzaPalatucciValdinoci}.  We include the main details in order to fix the precise form needed later.

The homogeneous energy space is
\[
\dot H^s(\R^N)= \{v\in L^{2_s^*}(\R^N):[v]_{H^s(\R^N)}<\infty\}.
\]
There is a constant $C_{\mathrm{Sob}}=C_{\mathrm{Sob}}(N,s)>0$ such that
\[
\|v\|_{L^{2_s^*}(\R^N)}^2\le C_{\mathrm{Sob}}[v]_{H^s(\R^N)}^2 \qquad\text{for every }v\in\dot H^s(\R^N).
\]

\begin{lemma}\label{Lem2.1}
Let \(K\subset\R^N\) be compact and satisfy \(\Cap_{s,2}(K)=0\). For every open neighborhood \(U\) of \(K\) and every \(\eps>0\), there exists \(\eta_\eps\in C^\infty(\R^N)\) such that
\[
0\le\eta_\eps\le1, \qquad \eta_\eps=0\text{ in a neighborhood of }K, \qquad \eta_\eps=1\text{ on }\R^N\setminus U,
\]
and
\[
\|1-\eta_\eps\|_{H^s(\R^N)}<\eps.
\]
\end{lemma}

\begin{proof}
Choose $\chi\in C_c^\infty(U)$ with $0\le\chi\le1$ and $\chi=1$ near $K$. By zero capacity, there are $\psi_j\in C_c^\infty(\R^N)$ with $\psi_j\ge1$ near $K$ and $\|\psi_j\|_{H^s}\to0$. Let $T\in C^\infty(\R)$ be Lipschitz, with $0\le T\le1$, $T=0$ on $(-\infty,1/2]$, and $T=1$ on $[1,\infty)$. Multiplication by $\chi$ is bounded on $H^s$, and $T(0)=0$, so
\[
\|T(\chi\psi_j)\|_{H^s} \le C_T\|\chi\psi_j\|_{H^s} \le C\|\psi_j\|_{H^s}\longrightarrow0.
\]
For large $j$, the function $\eta_\eps=1-T(\chi\psi_j)$ has all the stated properties.
\end{proof}

\begin{lemma}\label{Lem2.2}
If \(K\subset\R^N\) is compact and \(\Cap_{s,2}(K)=0\), then \(|K|=0\). Moreover, changing an \(H^s\)-function on \(K\) preserves its \(H^s\)-class. Finally, if \(v\in H^s(\R^N)\), then \(v\) may be approximated in \(H^s(\R^N)\) by functions vanishing in a neighborhood of \(K\).
\end{lemma}

\begin{proof}
Let \(\varphi\in C_c^\infty(\R^N)\) be admissible in the definition of \(\Cap_{s,2}(K)\). Since \(\varphi\ge1\) in a neighborhood of \(K\), the fractional Sobolev inequality yields
\[
|K|^{1/2_s^*} \le \|\varphi\|_{L^{2_s^*}(\R^N)} \le C[\varphi]_{H^s(\R^N)} \le C\|\varphi\|_{H^s(\R^N)}.
\]
Taking the infimum over all admissible \(\varphi\) gives \(|K|=0\). Therefore two functions that differ only on \(K\) define the same Lebesgue class, and hence the same \(H^s\)-class whenever one of them belongs to \(H^s(\R^N)\).

To prove the approximation assertion, let \(v\in H^s(\R^N)\). Choose \(g_j\in C_c^\infty(\R^N)\) such that \(g_j\to v\) in \(H^s(\R^N)\), as in \cite[Theorem 2.4]{DiNezzaPalatucciValdinoci}. For each fixed \(j\), let \(\eta_{\eps}\) be given by Lemma~\ref{Lem2.1}. Multiplication by \(g_j\) is continuous on \(H^s(\R^N)\), and hence
\[
\|g_j-g_j\eta_{\eps}\|_{H^s(\R^N)} = \|g_j(1-\eta_{\eps})\|_{H^s(\R^N)} \le C_j\|1-\eta_{\eps}\|_{H^s(\R^N)} \longrightarrow0.
\]
Since \(g_j\eta_{\eps}\in C_c^\infty(\R^N\setminus K)\), taking a diagonal sequence gives
\[
H^s(\R^N)=\overline{C_c^\infty(\R^N\setminus K)}^{\,H^s(\R^N)}.
\]
Thus sets of zero \((s,2)\)-capacity are removable for \(H^s\)-testing; see also \cite[Chapter 2]{AdamsHedberg}.
\end{proof}

\begin{lemma}\label{Lem2.3}
If \(R_\lambda\) denotes reflection with respect to \(T_\lambda=\{x_1=\lambda\}\), then
\[
\Cap_{s,2}(K)=0 \quad\Longrightarrow\quad \Cap_{s,2}(R_\lambda K)=0.
\]
\end{lemma}

\begin{proof}
Reflection preserves Lebesgue measure and distance, hence $\|\varphi\circ R_\lambda\|_{H^s}=\|\varphi\|_{H^s}$. If $\varphi\ge1$ near $K$, then $\varphi\circ R_\lambda\ge1$ near $R_\lambda K$. Taking infima in the capacity definition gives $\Cap_{s,2}(R_\lambda K)\le\Cap_{s,2}(K)$. Since $R_\lambda^2=\mathrm{Id}$, the reverse inequality also holds.
\end{proof}

\begin{lemma}\label{Lem2.4}
Let \(0<\mu<N\), \(1<a,b<+\infty\), and
\[
\frac1a+\frac1b+\frac\mu N=2.
\]
Then
\[
\int_{\R^N}\int_{\R^N} \frac{|F(x)|\,|G(y)|}{|x-y|^\mu}\dd x\dd y \le C\|F\|_{L^a(\R^N)}\|G\|_{L^b(\R^N)}.
\]
The same estimate holds over measurable subsets after extending functions by zero.
\end{lemma}

\begin{proof}
This is the classical Hardy-Littlewood-Sobolev inequality; see \cite[Theorem 4.3]{LiebLoss}. If the functions are originally defined only on measurable subsets, one extends them by zero to \(\R^N\) and applies the full-space estimate.
\end{proof}

\begin{proposition}\label{Prop2.1}
Let \(u\) be a positive weak singular solution satisfying condition \ref{Ass1.2}. Let \(E\subset\R^N\) be measurable and let \(Z\in\dot H^s(\R^N)\) be nonnegative with \(Z=0\) a.e. on \(\R^N\setminus E\). Then
\begin{align}\label{eq2.1}
\int_E\int_E \frac{u^{p-1}(y)Z(y)u^q(x)Z(x)}{|x-y|^\mu}\dd y\dd x\le C\|u\|_{L^{m_H}(E)}^{p+q-1} [Z]_{H^s(\R^N)}^2,
\end{align}
and
\begin{align}\label{eq2.2}
\int_E \left( \int_{\R^N}\frac{u^p(y)}{|x-y|^\mu}\dd y \right)u^{q-1}(x)Z^2(x)\dd x\le C\|u\|_{L^{m_H}(\R^N)}^p \|u\|_{L^{m_H}(E)}^{q-1} [Z]_{H^s(\R^N)}^2.
\end{align}
Estimate \eqref{eq2.2} remains valid, with the same constant, if \(u^p(y)\) is replaced by \(u_\lambda^p(y)\), since reflection preserves the \(L^{m_H}(\R^N)\)-norm.
\end{proposition}

\begin{proof}
We first prove \eqref{eq2.1}. Set
\[
\frac1{a_1}=\frac{p-1}{m_H}+\frac1{2_s^*}, \qquad \frac1{b_1}=\frac{q}{m_H}+\frac1{2_s^*}.
\]
By \eqref{eq1.4}, \(a_1,b_1\in(1,+\infty)\). Moreover, by \eqref{eq1.3},
\[
\frac1{a_1}+\frac1{b_1}+\frac\mu N = \frac{p+q-1}{m_H}+\frac2{2_s^*}+\frac\mu N =2.
\]
Applying Lemma~\ref{Lem2.4} to
\[
F(y)=u^{p-1}(y)Z(y)\one_E(y), \qquad G(x)=u^q(x)Z(x)\one_E(x),
\]
we obtain
\[
\int_E\int_E \frac{u^{p-1}(y)Z(y)u^q(x)Z(x)}{|x-y|^\mu}\dd y\dd x \le C\|u^{p-1}Z\|_{L^{a_1}(E)} \|u^qZ\|_{L^{b_1}(E)}.
\]
H\"older's inequality gives
\[
\|u^{p-1}Z\|_{L^{a_1}(E)} \le \|u\|_{L^{m_H}(E)}^{p-1} \|Z\|_{L^{2_s^*}(\R^N)},
\]
and
\[
\|u^qZ\|_{L^{b_1}(E)} \le \|u\|_{L^{m_H}(E)}^q \|Z\|_{L^{2_s^*}(\R^N)}.
\]
The fractional Sobolev inequality implies
\[
\|Z\|_{L^{2_s^*}(\R^N)} \le C[Z]_{H^s(\R^N)},
\]
and therefore \eqref{eq2.1} follows.

We next prove \eqref{eq2.2}. Set
\[
a_2=\frac{m_H}{p}, \qquad \frac1{b_2}=\frac{q-1}{m_H}+\frac2{2_s^*}.
\]
By \eqref{eq1.4}, \(a_2,b_2\in(1,+\infty)\), and
\[
\frac1{a_2}+\frac1{b_2}+\frac\mu N = \frac{p+q-1}{m_H}+\frac2{2_s^*}+\frac\mu N =2.
\]
Thus Lemma~\ref{Lem2.4}, applied to
\[
F(y)=u^p(y), \qquad G(x)=u^{q-1}(x)Z^2(x)\one_E(x),
\]
gives
\[
\int_E \left( \int_{\R^N}\frac{u^p(y)}{|x-y|^\mu}\dd y \right)u^{q-1}(x)Z^2(x)\dd x \le C\|u^p\|_{L^{a_2}(\R^N)} \|u^{q-1}Z^2\|_{L^{b_2}(E)}.
\]
Since \(a_2=m_H/p\),
\[
\|u^p\|_{L^{a_2}(\R^N)}= \|u\|_{L^{m_H}(\R^N)}^p.
\]
Again by H\"older's inequality,
\[
\|u^{q-1}Z^2\|_{L^{b_2}(E)} \le \|u\|_{L^{m_H}(E)}^{q-1} \|Z\|_{L^{2_s^*}(\R^N)}^2 \le C\|u\|_{L^{m_H}(E)}^{q-1}[Z]_{H^s(\R^N)}^2.
\]
This proves \eqref{eq2.2}. Finally, if \(u^p\) is replaced by \(u_\lambda^p\), then
\[
\|u_\lambda\|_{L^{m_H}(\R^N)}= \|u\|_{L^{m_H}(\R^N)},
\]
because reflection is measure preserving. Hence the same proof gives the reflected estimate with the same constant.
\end{proof}

\begin{lemma}\label{Lem2.5}
Let \(u\) be a positive weak singular solution satisfying condition \ref{Ass1.2}. Then \(\Hs[u]\) is locally bounded in \(\R^N\setminus\Gamma\). In fact, if
\[
\beta_B= \min\left\{1,\,N-\mu-\frac{Np}{m_B}\right\},
\]
then \(\beta_B>0\) and
\[
\Hs[u]\in C^\beta_{\loc}(\R^N\setminus\Gamma) \qquad\text{for every }0<\beta<\beta_B.
\]
\end{lemma}

\begin{proof}
Since \(\Hs[u]\) is finite a.e. in \(\R^N\setminus\Gamma\), we may choose \(x_*\in\R^N\setminus\Gamma\) such that \({\Hs[u](x_*)<+\infty}\). Let \(R_*>2|x_*|+2\). For \(|y|>R_*\) one has
\[
1+|y|^\mu\ge c|x_*-y|^\mu,
\]
and hence
\[
\int_{(\R^N\setminus\Gamma)\setminus B_{R_*}} \frac{u^p(y)}{1+|y|^\mu}\dd y \le C\Hs[u](x_*)<+\infty.
\]
On \(B_{R_*}\setminus\Gamma\), the same finiteness gives \(u^p\in L^1(B_{R_*}\setminus\Gamma)\), because \(|x_*-y|\) is bounded above there and
\[
u^p(y)\le C_{R_*}\frac{u^p(y)}{|x_*-y|^\mu} \quad\text{for a.e. }y\in B_{R_*}\setminus(\Gamma\cup\{x_*\}).
\]
Consequently,
\[
\int_{\R^N\setminus\Gamma}\frac{u^p(y)}{1+|y|^\mu}\dd y<+\infty.
\]

Now let \(K\Subset \R^N\setminus\Gamma\). Choose \(\rho>0\) such that
\[
U_\rho(K)=\{y\in\R^N:\dist(y,K)<\rho\}\Subset\R^N\setminus\Gamma.
\]
By \eqref{eq1.5} and \eqref{eq1.6}, \(u\in L^{m_B}(\R^N)\) and
\[
a:=\frac{m_B}{p}>\frac{N}{N-\mu}.
\]
Let \(a'\) be the conjugate exponent of \(a\). Then \(\mu a'<N\). For \(x\in K\), split
\[
\Hs[u](x)= \int_{U_\rho(K)}\frac{u^p(y)}{|x-y|^\mu}\dd y + \int_{(\R^N\setminus\Gamma)\setminus U_\rho(K)} \frac{u^p(y)}{|x-y|^\mu}\dd y =:J_1(x)+J_2(x).
\]
For the local part, H\"older's inequality gives
\[
J_1(x) \le \|u^p\|_{L^a(U_\rho(K))} \left( \int_{U_\rho(K)}\frac{\dd y}{|x-y|^{\mu a'}} \right)^{1/a'}.
\]
Since \(\mu a'<N\), the last integral is uniformly bounded for \(x\in K\). Therefore
\[
\sup_{x\in K}J_1(x)<+\infty.
\]

For the nonlocal part, choose \(R>1\) so large that \(U_\rho(K)\subset B_R\). If \(x\in K\) and \(y\in B_R\setminus U_\rho(K)\), then \(|x-y|\ge\rho\), and hence
\[
\int_{B_R\setminus(\Gamma\cup U_\rho(K))} \frac{u^p(y)}{|x-y|^\mu}\dd y \le \rho^{-\mu}\int_{B_R\setminus\Gamma}u^p(y)\dd y<+\infty.
\]
Increasing \(R\), if necessary, we also have \(|x-y|\ge |y|/2\) for \(x\in K\) and \(|y|>R\). Therefore
\[
\int_{(\R^N\setminus\Gamma)\setminus B_R} \frac{u^p(y)}{|x-y|^\mu}\dd y \le C\int_{\R^N\setminus\Gamma} \frac{u^p(y)}{1+|y|^\mu}\dd y<+\infty.
\]
Combining the preceding estimates gives \(\sup_K\Hs[u]<+\infty\).

For the H\"older refinement, put
\[
F=u^p\one_{\R^N\setminus\Gamma}, \qquad a=\frac{m_B}{p}, \qquad \delta=N-\mu-\frac Na>0.
\]
Then \(F\in L^a(\R^N)\), and the weighted tail estimate proved above holds for \(F\). For \(x,z\) in a fixed compact subset of \(\R^N\setminus\Gamma\), set \(h=|x-z|\). Splitting a fixed bounded integration region into the \(2h\)-neighborhoods of \(x,z\) and their complement, H\"older's inequality and the mean value theorem give
\[
\int_{B_R}F(y) \left| \frac1{|x-y|^\mu}-\frac1{|z-y|^\mu} \right|\dd y \le C
\begin{cases}
h^\delta, & 0<\delta<1,\\
h(1+|\log h|)^{1/a'}, & \delta=1,\\
h, & \delta>1.
\end{cases}
\]
For \(R\) sufficiently large, the mean value theorem also gives
\[
\int_{(\R^N\setminus\Gamma)\setminus B_R}F(y) \left| \frac1{|x-y|^\mu}-\frac1{|z-y|^\mu} \right|\dd y \le Ch\int_{\R^N\setminus\Gamma} \frac{u^p(y)}{1+|y|^\mu}\dd y.
\]
Thus \(\Hs[u]\in C^\beta_{\loc}\) for every \(0<\beta<\min\{1,\delta\}=\beta_B\). Since \(K\Subset\R^N\setminus\Gamma\) was arbitrary, the proof is complete.
\end{proof}

\begin{proposition}\label{Prop2.2}
Let \(u\) be a positive weak singular solution satisfying condition \ref{Ass1.2}. Set
\[
r_B=\frac{m_B}{q}, \qquad \alpha_B=\min\left\{2s-\frac{N}{r_B},1\right\}.
\]
Then
\[
\Hs[u]u^q\in L^{r_B}_{\loc}(\R^N\setminus\Gamma), \qquad r_B>\frac{N}{2s}.
\]
Consequently, \(u\) has a representative satisfying
\[
u\in C^\alpha_{\loc}(\R^N\setminus\Gamma)
\]
for every \(0<\alpha<\alpha_B\). This representative is unique. Set
\[
\gamma_B=\min\{\alpha_B,\beta_B\},
\]
where \(\beta_B\) is defined in Lemma~\ref{Lem2.5}. For every \(0<\gamma<\gamma_B\), this representative further satisfies
\[
\Hs[u]u^q\in C^\gamma_{\loc}(\R^N\setminus\Gamma).
\]
Consequently, whenever \(2s+\gamma\notin\N\),
\[
u\in C^{2s+\gamma}_{\loc}(\R^N\setminus\Gamma),
\]
with the usual noninteger H\"older-space interpretation.
\end{proposition}

\begin{proof}
Let \(\Omega\Subset\R^N\setminus\Gamma\). By Lemma~\ref{Lem2.5},
\[
\Hs[u]\in L^\infty(\Omega).
\]
By \eqref{eq1.5} and \eqref{eq1.6}, \(u\in L^{m_B}(\Omega)\), and therefore
\[
u^q\in L^{m_B/q}(\Omega).
\]
The condition \(m_B>qN/(2s)\) in \eqref{eq1.5} gives
\[
r_B>\frac{N}{2s}.
\]
Thus
\[
\Hs[u]u^q\in L^{r_B}(\Omega)
\]
for every \(\Omega\Subset\R^N\setminus\Gamma\), which proves the first assertion.

In every \(\Omega\Subset\R^N\setminus\Gamma\), the function \(u\) satisfies
\[
(-\Delta)^s u=f \quad\text{in }\Omega, \qquad f:=\Hs[u]u^q\in L^{r_B}(\Omega), \qquad r_B>\frac{N}{2s}.
\]
Fix
\[
\Omega'\Subset\Omega_1\Subset\Omega
\]
and choose \(\zeta\in C_c^\infty(\Omega)\) with \(\zeta\equiv1\) on \(\Omega_1\). Extend \(\zeta f\) by zero outside \(\Omega\), and define its Riesz potential
\[
P(x)= \kappa_{N,s}\int_{\R^N} \frac{\zeta(y)f(y)}{|x-y|^{N-2s}}\dd y.
\]
Here $\kappa_{N,s}=c_{N,2s}$ is normalized so that $(-\Delta)^s I_{2s}f=f$. Since \(\zeta f\in L^{r_B}(\R^N)\) has compact support and \(r_B>N/(2s)\), the standard Riesz-potential estimate gives
\[
P\in C^\alpha_{\loc}(\R^N) \qquad\text{for every } 0<\alpha<\min\left\{2s-\frac{N}{r_B},1\right\}.
\]
Moreover, \(P\in\Ls(\R^N)\) and
\[
(-\Delta)^sP=\zeta f \quad\text{in }\mathcal D'(\R^N).
\]
Consequently \(h:=u-P\) belongs to \(\Ls(\R^N)\) and is \(s\)-harmonic in \(\Omega_1\). Interior regularity for \(s\)-harmonic functions \cite[Section 2]{Fall2016} gives \(h\in C^\infty(\Omega')\). Hence
\[
u=P+h\in C^\alpha(\Omega')
\]
for every \(0<\alpha<\alpha_B\). This gives the desired local H\"older representative. If two continuous representatives agree almost everywhere, they agree everywhere, which proves uniqueness.

Finally, fix \(0<\gamma<\gamma_B\). Lemma~\ref{Lem2.5} and the first part of the proposition give
\[
\Hs[u]\in C^\gamma_{\loc}, \qquad u\in C^\gamma_{\loc}.
\]
The continuous representative is nonnegative everywhere, and \(t\mapsto t^q\) is locally Lipschitz on \([0,+\infty)\). Therefore \(f=\Hs[u]u^q\in C^\gamma_{\loc}\). With the same cutoff as above, $\zeta f\in C_c^\gamma(\R^N)$. We retain the already defined potential $P$ and smooth remainder $h=u-P$. The interior Schauder estimate \cite[Corollary 2.4]{RosOtonSerra} (see also the potential argument in \cite[Proposition 2.8]{Silvestre}) gives
\[
P\in C^{2s+\gamma}(\Omega') \qquad\text{whenever }2s+\gamma\notin\N.
\]
To apply its smooth formulation, mollify $\zeta f$ and the corresponding potential $P$. The weighted $L^1$-norms and local $C^\gamma$-norms of the mollified potentials, as well as the $C^\gamma$-norms of their sources, are uniformly bounded; passage to the limit gives the assertion. Since $h$ is smooth in $\Omega'$, we conclude that $u\in C^{2s+\gamma}_{\loc}(\R^N\setminus\Gamma)$.

\end{proof}

\begin{lemma}\label{Lem2.6}
Let \(u\) be a positive weak singular solution satisfying condition \ref{Ass1.2}, and define
\begin{equation}\label{eq2.3}
\mathcal M(E)= \|u\|_{L^{m_H}(E)}^{p+q-1} + \|u\|_{L^{m_H}(\R^N)}^p \|u\|_{L^{m_H}(E)}^{q-1}.
\end{equation}
Then the following hold:
\begin{enumerate}[label=\textup{(\roman*)}]
\item if \(E_j\subset B_R\) and \(|E_j|\to0\), then \(\mathcal M(E_j)\to0\);
\item \(\mathcal M(\{x_1<\lambda\})\to0\) as \(\lambda\to-\infty\);
\item \(\mathcal M(\{x_1>\lambda\})\to0\) as \(\lambda\to+\infty\);
\item \(\mathcal M(\R^N\setminus B_R)\to0\) as \(R\to+\infty\).
\end{enumerate}
\end{lemma}

\begin{proof}
All four assertions follow from absolute continuity of the integral of $u^{m_H}$. In \textup{(i)}, the measures of the sets tend to zero; in \textup{(ii)}-\textup{(iv)}, the sets decrease to the empty set. Thus the corresponding $L^{m_H}$-norms tend to zero. Both powers $p+q-1$ and $q-1$ in \eqref{eq2.3} are positive.
\end{proof}

\begin{lemma}\label{Lem2.7}
Let \(u\) be a positive weak singular solution satisfying condition \ref{Ass1.2}. Set \(V=\Hs[u]\). Then \(V\in L^\infty(\R^N)\) and \(Vu^q\in L^{r_B}(\R^N)\), where \(r_B=m_B/q>N/(2s)\). For every \(\delta>0\),
\[
\sup_{\dist(x,\Gamma)\ge\delta}u(x)<\infty, \qquad u\in L^{2_s^*}\bigl(\{x:\dist(x,\Gamma)\ge\delta\}\bigr).
\]
Moreover, \(u(x)\to0\) as \(|x|\to\infty\).
\end{lemma}

\begin{proof}
Put \(a=m_B/p\) and \(b=m_H/p\). By \eqref{eq1.4}-\eqref{eq1.5}, \(a,b>1\), \(\mu a'<N\), and \(\mu b'>N\). For the last inequality, the HLS balance and \((q-1)/m_H<2s/N\) give
\[
\frac p{m_H}+\frac\mu N =1+\frac{2s}{N}-\frac{q-1}{m_H}>1.
\]
Splitting the Hartree integral into \(|x-y|<1\) and \(|x-y|\ge1\), and applying H\"older's inequality with exponents \(a\) and \(b\), respectively, yields
\[
\|V\|_{L^\infty(\R^N)} \le C\bigl(\|u\|_{L^{m_B}(\R^N)}^p +\|u\|_{L^{m_H}(\R^N)}^p\bigr).
\]
Consequently \(f:=Vu^q\in L^{r_B}(\R^N)\).

Fix \(\delta>0\) and set \(\rho=\min\{1,\delta/4\}\). The local Riesz-potential decomposition used in Proposition \ref{Prop2.2}, together with the interior mean-value estimate for \(s\)-harmonic functions \cite[Section 2]{Fall2016}, gives, for \(\dist(x,\Gamma)\ge\delta\),
\[
\sup_{B_{\rho/2}(x)}u \le C_\delta\left( \|f\|_{L^{r_B}(B_{2\rho}(x))} +\int_{\R^N}\frac{u(y)}{1+|x-y|^{N+2s}}\dd y \right).
\]
Indeed, choose $\zeta_x\in C_c^\infty(B_{2\rho}(x))$, with $0\le\zeta_x\le1$ and equal to one on $B_\rho(x)$. Let $P_x$ be the Riesz potential of $\zeta_xf$, with the normalization used in Proposition~\ref{Prop2.2}, and put $h_x=u-P_x$. Then $P_x\ge0$,
\[
\|P_x\|_{L^\infty(\R^N)} \le C\rho^{2s-N/r_B}\|f\|_{L^{r_B}(B_{2\rho}(x))},
\]
and $h_x\in\Ls(\R^N)$ is $s$-harmonic in $B_\rho(x)$. Choose $\phi\in C_c^\infty((1,4))$ with $\phi\ge0$ and $\int\phi=1$, and define the averaged ball Poisson kernel
\[
\Phi(z)=\beta_{N,s}|z|^{-N} \int_1^{\min\{4,|z|\}} \frac{t^{2s}\phi(t)}{(|z|^2-t^2)^s}\,\dd t \quad (|z|>1), \qquad \Phi(z)=0\quad (|z|\le1),
\]
where $\beta_{N,s}>0$ is the normalizing constant in \cite{Fall2016}. For $\Phi_\ell(z)=\ell^{-N}\Phi(z/\ell)$ and $\ell=\rho/16$, the mean-value formula and kernel estimates in \cite{Fall2016} yield
\[
h_x=h_x*\Phi_\ell\quad\text{in }B_{\rho/2}(x), \qquad 0\le\Phi_\ell(z)\le \frac{C\ell^{2s}}{(\ell+|z|)^{N+2s}}.
\]
Hence, for $z\in B_{\rho/2}(x)$, positivity of $P_x$ gives
\[
u(z)\le \|P_x\|_\infty+ \int_{\R^N}u(y)\Phi_\ell(z-y)\,\dd y \le \|P_x\|_\infty+ C_\delta\int_{\R^N} \frac{u(y)}{1+|x-y|^{N+2s}}\,\dd y,
\]
which proves the estimate. Both terms on its right are uniformly bounded. They also tend to zero as \(|x|\to\infty\): this follows from \(f\in L^{r_B}(\R^N)\) for the first term, and from \(u\in L^{m_H}(\R^N)\) and H\"older's inequality for the second, by first splitting at a fixed large ball. This proves the boundedness and decay assertions. Finally, \(m_H<2_s^*\), so boundedness on \(\{\dist(\cdot,\Gamma)\ge\delta\}\) and the global \(L^{m_H}\)-bound give the asserted \(L^{2_s^*}\)-integrability.
\end{proof}

\begin{proposition}\label{Prop2.3}
Let
\[
N=2,\qquad s=\frac12,\qquad \mu=1,\qquad p=q=2, \qquad \Gamma=\{0\}\times[-1,1].
\]
Then \(\Cap_{1/2,2}(\Gamma)=0\), and \eqref{eq1.1} admits a positive weak singular solution
\[
u\in L^3(\R^2)\cap L^5(\R^2)
\]
having a nonremovable singularity at every point of \(\Gamma\). In particular, the solution classes in Theorems~\ref{Thm1.1}-\ref{Thm1.3} are nonempty.
\end{proposition}

\begin{proof}
We first verify the capacity assertion. Choose \(\eta\in C_c^\infty(\R)\) equal to one on \([-2,2]\), and, for \(0<\rho<1/2\), put \(L_\rho=\log(1/\rho)\) and
\[
h_\rho(r)=
\begin{cases}
1,&0\le r\le\rho,\\
\log(1/r)/L_\rho,&\rho<r<1,\\
0,&r\ge1.
\end{cases}
\]
The function \(U_\rho(x_1,x_2,t)=\eta(x_2)h_\rho(\sqrt{x_1^2+t^2})\), defined for \(t>0\), has trace equal to one near \(\Gamma\). Integration in polar coordinates in the \((x_1,t)\)-half-plane gives
\[
\int_{\R^3_+}\bigl(|\nabla U_\rho|^2+|U_\rho|^2\bigr) \le C\bigl(L_\rho^{-1}+L_\rho^{-2}+\rho^2\bigr)\longrightarrow0;
\]
indeed, the radial term gives \(\int_\rho^1|h_\rho'(r)|^2r\dd r=L_\rho^{-1}\), whereas \(\int_0^1h_\rho(r)^2r\dd r\le C{(L_\rho^{-2}+\rho^2)}\). The trace inequality gives convergence to zero in \(H^{1/2}(\R^2)\). Convolving each trace with a sufficiently small nonnegative smooth kernel preserves the value one on a smaller neighborhood of \(\Gamma\) and does not increase this norm. These are admissible functions in the definition of \(\Cap_{1/2,2}\), proving the~assertion.

Write
\[
Th(x)=\int_{\R^2}\frac{h(y)}{|x-y|}\dd y, \qquad I_1h=\kappa Th,
\]
where \(\kappa>0\) is normalized so that \((-\Delta)^{1/2}I_1h=h\). If \(\sigma=\mathcal H^1\!\restriction_\Gamma\), its Riesz potential is
\[
W(x)=I_1\sigma(x) =\kappa\int_{-1}^1\frac{dt}{\sqrt{x_1^2+(x_2-t)^2}}.
\]
This function is positive and smooth outside \(\Gamma\). Directly from the integral, it satisfies
\[
W(x)\ge\frac{c}{1+|x|},\qquad W(x)\le\frac{C}{|x|}\quad(|x|\ge4),
\]
and, on \(B_4\setminus\Gamma\), \(W(x)\le C(1+\log_+(1/\dist(x,\Gamma)))\), where \(\log_+t=\max\{\log t,0\}\). The area of the \(r\)-neighborhood of \(\Gamma\) is \(O(r)\) as \(r\downarrow0\). Consequently \(W\in L^a(\R^2)\) for every finite \(a>2\), and \(\int_{B_R}W^2\le C\log(2+R)\).

Set \(V_W=T(W^2)\) and \(F=V_WW^2\). Splitting the kernel at radius one, H\"older's inequality with \(W^2\in L^3\) in the near region and \(W^2\in L^{3/2}\) in the far region shows that \(V_W\in L^\infty(\R^2)\). For \(r=|x|\ge8\), split the integral into \(|y|\le r/2\), \(|y-x|\le r/2\), and the remaining region. The first contribution is at most \(C\log(2+r)/r\); the second is at most \(C/r\), using \(W(y)^2\le C/r^2\). Splitting the remaining region further at \(|y|=2r\) bounds both contributions by \(C/r\). It follows that
\[
V_W(x)\le C\frac{\log(2+|x|)}{|x|},\qquad F(x)\le C\frac{\log(2+|x|)}{|x|^3} \quad(|x|\ge8).
\]
Thus \(F\in L^1(\R^2)\cap L^a(\R^2)\) for every finite \(a\ge1\). In particular, \(I_1F\) is bounded, by splitting the kernel at radius one and using \(F\in L^3\cap L^1\). For \(r=|x|\ge16\), the same decomposition as above gives a contribution at most \(C\|F\|_1/r\) from \(|y|\le r/2\). The other three contributions are each at most \(C\log(2+r)/r^2\), by the displayed bound on \(F\). Combining this decay with the lower bound for \(W\), we obtain
\begin{equation}\label{eq2.4}
I_1\bigl[(T(W^2))W^2\bigr]\le C_0W \qquad\text{in }\R^2\setminus\Gamma.
\end{equation}

Choose \(\eps>0\) so that \(16C_0\eps^3\le1\), and define
\[
u_0=\eps W,\qquad u_{n+1}=\eps W+I_1\bigl[(T(u_n^2))u_n^2\bigr].
\]
Positivity of the kernels and \eqref{eq2.4} imply inductively that \(\eps W\le u_n\le u_{n+1}\le2\eps W\). The monotone convergence theorem, applied to both integrals, yields a finite limit outside \(\Gamma\) satisfying
\[
u=\eps W+I_1f,\qquad f=(T(u^2))u^2,\qquad \eps W\le u\le2\eps W.
\]
In particular, \(T(u^2)\le4\eps^2V_W\) and \(0\le f\le16\eps^4F\), so the potential is bounded and \(f\in L^1\cap L^a\) for every finite \(a\ge1\). The fundamental-solution identity gives \((-\Delta)^{1/2}u=\eps\sigma+f\) in distributions on \(\R^2\), hence \eqref{eq1.1} holds for all tests supported outside \(\Gamma\). Since \(f\in L^2\), the Riesz transform multipliers \(i\xi_j/|\xi|\) give \(\nabla I_1f\in L^2\). Also \(I_1f\) is bounded, and therefore belongs to \(H^1_{\loc}(\R^2)\). Together with the smoothness of \(W\) outside \(\Gamma\), this proves \(u\in H^{1/2}_{\loc}(\R^2\setminus\Gamma)\). The bounds by \(W\) also give \(u\in\Ls(\R^2)\cap L^3(\R^2)\cap L^5(\R^2)\). For every $z\in\Gamma$, a subsegment of length one starting at $z$ is contained in $\Gamma$. The triangle inequality therefore gives
\[
W(x)\ge\kappa\int_0^1\frac{dt}{|x-z|+t} =\kappa\log\left(1+\frac1{|x-z|}\right).
\]
Since $u\ge\eps W$, every point of $\Gamma$ is nonremovable in the sense of the stated definition.

For these parameters, \(2_s^*=4\), \(m_H=3\), and the maximum in \eqref{eq1.4} equals \(8/3\). Moreover, \(m_B=5>4\) satisfies \eqref{eq1.5}, while \(p+q+1=5<6=2\,2_{\mu,s}^*\). The bounds in \eqref{eq1.9} are \(p=2<3\) and \(q=2\le2\). All assumptions of Theorems~\ref{Thm1.1}-\ref{Thm1.3} are therefore satisfied.
\end{proof}

\section{Test Functions near the Reflected Singular Set}\label{Sec3}

Unless otherwise stated, throughout Sections~\ref{Sec3}-\ref{Sec5} we assume \ref{Ass1.1} and \ref{Ass1.2} and use the continuous representative fixed in the introduction. We retain the reflection notation introduced there and take $\lambda<0$. The half-space \(\Sigma_\lambda\) has positive distance from \(\Gamma\), whereas \(\Gamma_\lambda\subset\Sigma_\lambda\). We first use local regularity to justify compactly supported test functions in the reflected equation. The right moving plane uses the same arguments in \(\{x_1>\lambda\}\), with \(\lambda>0\).

For compactly supported tests away from the singular sets, write
\[
\E(h,\psi)=\frac{C_{N,s}}2 \iint_{\R^N\times\R^N} \frac{(h(x)-h(y))(\psi(x)-\psi(y))}{|x-y|^{N+2s}}\dd x\dd y,
\]
where \(C_{N,s}\) is the normalization constant for \((-\Delta)^s\). For a bounded smooth open set \(\Omega\), we use
\[
H_0^s(\Omega)=\{\psi\in H^s(\R^N):\psi=0\text{ a.e. in }\R^N\setminus\Omega\}.
\]

\begin{lemma}\label{Lem3.1}
Let \(u\) be a positive weak singular solution of \eqref{eq1.1}, and let \(\lambda<0\). Then \(u_\lambda\) is a positive weak singular solution of
\[
(-\Delta)^s u_\lambda=V_\lambda u_\lambda^q \quad\text{in }\R^N\setminus\Gamma_\lambda
\]
in the weak sense, where
\[
V_\lambda(x)= \int_{\R^N\setminus\Gamma_\lambda} \frac{u_\lambda^p(y)}{|x-y|^\mu}\dd y = \Hs[u](x^\lambda) \quad\text{for a.e. }x\in\R^N\setminus\Gamma_\lambda.
\]
\end{lemma}

\begin{proof}
The affine reflection preserves Lebesgue measure and maps \(\R^N\setminus\Gamma\) bijectively onto \(\R^N\setminus\Gamma_\lambda\). Hence, for a.e. \(x\in\R^N\setminus\Gamma_\lambda\), the change of variables \(y=z^\lambda\) gives
\[
V_\lambda(x) = \int_{\R^N\setminus\Gamma_\lambda} \frac{u_\lambda^p(y)}{|x-y|^\mu}\dd y = \int_{\R^N\setminus\Gamma} \frac{u^p(z)}{|x-z^\lambda|^\mu}\dd z.
\]
Since reflection is an isometry,
\[
|x-z^\lambda|=|x^\lambda-z|,
\]
and therefore
\[
V_\lambda(x)=\int_{\R^N\setminus\Gamma}\frac{u^p(z)}{|x^\lambda-z|^\mu}\dd z =\Hs[u](x^\lambda).
\]

We also verify the requirements in Definition~\ref{Def1.2}. Since \(1+|x^\lambda|^{N+2s}\asymp_\lambda 1+|x|^{N+2s}\), we have \(u_\lambda\in\Ls(\R^N)\). Positivity a.e. and local \(H^s\)-regularity away from \(\Gamma_\lambda\) are inherited from \(u\). Moreover,
\[
V_\lambda u_\lambda^q=(\Hs[u]u^q)\circ R_\lambda \quad\text{a.e. in }\R^N\setminus\Gamma_\lambda,
\]
so the reflected potential is finite a.e. and the reflected source belongs to \(L^1_{\loc}(\R^N\setminus\Gamma_\lambda)\). More generally, since reflection preserves Lebesgue measure,
\[
\|u_\lambda\|_{L^m(\R^N)}=\|u\|_{L^m(\R^N)} \qquad\text{for every }1\le m\le+\infty
\]
whenever either side is finite. In particular, condition \ref{Ass1.2} and its \(L^{m_H}\)- and \(L^{m_B}\)-bounds are~preserved.

We next verify the weak equation. Let \(\varphi\in C_c^\infty(\R^N\setminus\Gamma_\lambda)\), and define
\[
\varphi^\lambda(x):=\varphi(x^\lambda).
\]
Then \(\varphi^\lambda\in C_c^\infty(\R^N\setminus\Gamma)\). Since reflection is an isometry, for every \(x\in\R^N\),
\[
(-\Delta)^s\varphi^\lambda(x) = C_{N,s}\,\mathrm{P.V.}\int_{\R^N} \frac{\varphi(x^\lambda)-\varphi(y^\lambda)}{|x-y|^{N+2s}}\dd y.
\]
Making the change of variables \(z=y^\lambda\), we obtain
\[
(-\Delta)^s\varphi^\lambda(x) = C_{N,s}\,\mathrm{P.V.}\int_{\R^N} \frac{\varphi(x^\lambda)-\varphi(z)}{|x^\lambda-z|^{N+2s}}\dd z = ((-\Delta)^s\varphi)(x^\lambda).
\]
Therefore, using the weak formulation of \(u\) and then the change of variables \(x=y^\lambda\),
\begin{align*}
\int_{\R^N}u_\lambda(x)(-\Delta)^s\varphi(x)\dd x &= \int_{\R^N}u(y)\,((-\Delta)^s\varphi)(y^\lambda)\dd y \\
&= \int_{\R^N}u(y)(-\Delta)^s\varphi^\lambda(y)\dd y \\
&= \int_{\R^N\setminus\Gamma}\Hs[u](y)u^q(y)\varphi^\lambda(y)\dd y \\
&= \int_{\R^N\setminus\Gamma_\lambda} \Hs[u](x^\lambda)u_\lambda^q(x)\varphi(x)\dd x \\
&= \int_{\R^N\setminus\Gamma_\lambda} V_\lambda(x)u_\lambda^q(x)\varphi(x)\dd x.
\end{align*}
This proves the lemma.
\end{proof}

\begin{lemma}\label{Lem3.2}
Let \(\Gamma\subset\{x_1=0\}\) be compact, let \(\lambda<0\), and let \(u\in\Ls(\R^N)\cap H^s_{\loc}(\R^N\setminus\Gamma)\). If \(\xi\in C_c^\infty(\R^N)\), \(0\le\xi\le1\), vanishes near \(\Gamma_\lambda\), then
\[
W_\lambda\xi\in H^s(\R^N), \qquad W_\lambda\xi^2\in H_0^s(\Omega)
\]
for some bounded smooth open set \(\Omega\Subset\R^N\setminus(\Gamma\cup\Gamma_\lambda)\). Both pairings \(\E(u,W_\lambda\xi^2)\) and \(\E(u_\lambda,W_\lambda\xi^2)\) are absolutely convergent. If, in addition, \((-\Delta)^su=f\) weakly away from \(\Gamma\), with \(f\in L^\infty_{\loc}(\R^N\setminus\Gamma)\), and \(f_\lambda=f\circ R_\lambda\), then
\begin{equation}\label{eq3.1}
\E(w_\lambda,W_\lambda\xi^2) =\int_{\R^N}(f-f_\lambda)W_\lambda\xi^2\dd x.
\end{equation}
In particular, this applies with \(f=Vu^q\) under condition \ref{Ass1.2}, and also to equations with locally bounded weighted sources.
\end{lemma}

\begin{proof}
Put \(K=\supp\xi\cap\overline{\Sigma_\lambda}\). Then \(K\cup R_\lambda K\) is compact and disjoint from \(\Gamma\cup\Gamma_\lambda\). Choose a smooth compactly supported function \(\chi\) in the complement of these sets, invariant under \(R_\lambda\), and equal to one near \(K\cup R_\lambda K\). The function \(g=\chi w_\lambda\) belongs to \(H^s(\R^N)\) and is antisymmetric. We claim that \(G=g^+\one_{\Sigma_\lambda}\in H^s(\R^N)\). The contribution over \(\Sigma_\lambda\times\Sigma_\lambda\) is controlled by the seminorm of \(g\), and \(\|G\|_2\le\|g\|_2\). For the cross interaction write the exterior point as \(z^\lambda\), with \(z\in\Sigma_\lambda\). Whenever \(G(x)>0\),
\[
G(x)^2\le
\begin{cases}
|g(x)-g(z^\lambda)|^2,&g(z)\ge0,\\
|g(x)-g(z)|^2,&g(z)<0.
\end{cases}
\]
In the second case \(|x-z|<|x-z^\lambda|\). Integrating the two cases and adding the contribution over $\Sigma_\lambda\times\Sigma_\lambda$ gives $[G]_{H^s}^2\le3[g]_{H^s}^2$, which proves the claim. Since \(G=W_\lambda\) on \(K\), multiplication by \(\xi\) and by \(\xi^2\) shows that both products belong to \(H^s(\R^N)\). Their supports are compact and disjoint from both singular sets. Choose \(\Omega\) with both supports compactly contained in it, and choose \(\overline\Omega\subset\Omega'\Subset \R^N\setminus(\Gamma\cup\Gamma_\lambda)\).

For any \(h\in\Ls(\R^N)\cap H^s(\Omega')\) and \(\psi\in H^s(\R^N)\) supported in \(\overline\Omega\), the integral of
\[
\frac{|h(x)-h(y)|\,|\psi(x)-\psi(y)|}{|x-y|^{N+2s}}
\]
over \(\Omega'\times\Omega'\) is bounded by \([h]_{H^s(\Omega')}[\psi]_{H^s}\). Since $\overline\Omega$ has positive distance from $\R^N\setminus\Omega'$, the remaining interactions are bounded by
\[
C\|h\|_{L^2(\Omega)}\|\psi\|_{L^2(\Omega)} +C\|\psi\|_{L^1(\Omega)} \int_{\R^N}\frac{|h(y)|}{1+|y|^{N+2s}}\dd y.
\]
Thus the pairing is absolutely convergent and continuous in \(\psi\) for the \(H^s\)-norm. Reflection preserves the weighted tail class because \(1+|x^\lambda|^{N+2s}\asymp_\lambda 1+|x|^{N+2s}\), so this applies to both \(u\) and \(u_\lambda\). For smooth tests, splitting into the same local and tail regions gives \(\E(u,\psi)=\int u(-\Delta)^s\psi\). The reflected weak equation has source \(f_\lambda\), by the change of variables used in Lemma~\ref{Lem3.1}. Approximate \(W_\lambda\xi^2\) in \(H^s(\R^N)\) by functions in \(C_c^\infty(\Omega)\). The pairings converge by the preceding estimates, and the source integrals converge in view of local boundedness and \(L^1(\Omega)\)-convergence. Subtracting the weak equations proves \eqref{eq3.1}.
\end{proof}

\begin{lemma}\label{Lem3.3}
Let \(w\) be a measurable function that is finite almost everywhere and antisymmetric with respect to \(T_\lambda\), and put \(W=w^+\one_{\Sigma_\lambda}\). Let \(0\le\xi\le1\) be measurable. Assume that \(W\xi\in H^s(\R^N)\), that the pairing \(\E(w,W\xi^2)\) is absolutely convergent, and that
\[
\mathcal R_\xi= \iint_{\Sigma_\lambda\times\Sigma_\lambda} \frac{W(x)W(y)|\xi(x)-\xi(y)|^2}{|x-y|^{N+2s}}\dd x\dd y<+\infty.
\]
Then
\begin{equation}\label{eq3.2}
\E(w,W\xi^2)\ge \frac{C_{N,s}}2\bigl([W\xi]_{H^s(\R^N)}^2-\mathcal R_\xi\bigr).
\end{equation}
\end{lemma}

\begin{proof}
Write \(\Sigma=\Sigma_\lambda\), \(U=w^-\one_\Sigma\), \(B=W\xi^2\), and \(Z=W\xi\). Splitting the pairing into the two half-spaces, reflecting the exterior variable, and symmetrizing gives
\begin{align*}
\frac{2}{C_{N,s}}\E(w,B) ={}&\iint_{\Sigma\times\Sigma} \frac{(w(x)-w(y))(B(x)-B(y))}{|x-y|^{N+2s}}\dd x\dd y\\
&+\iint_{\Sigma\times\Sigma} \frac{(w(x)+w(y))(B(x)+B(y))}{|x-y^\lambda|^{N+2s}}\dd x\dd y.
\end{align*}
Both displayed integrands are absolutely integrable. We justify the following expansions by first integrating on the increasing sets
\[
\begin{split}
D_n=\bigl\{(x,y)\in(\Sigma\cap B_n)^2:\;& \min\{\dist(x,T_\lambda),\dist(y,T_\lambda)\}>1/n,\\
&|x-y|>1/n,\quad |w(x)|+|w(y)|\le n\bigr\}.
\end{split}
\]
All terms there are integrable. Since \(w=W-U\) on \(\Sigma\) and \(WU=0\), the mixed terms in the two kernels combine into
\[
\bigl(U(x)W(y)\xi(y)^2+U(y)W(x)\xi(x)^2\bigr) \left(\frac1{|x-y|^{N+2s}}- \frac1{|x-y^\lambda|^{N+2s}}\right)\ge0.
\]
For the terms containing only \(W\) in the first double integral, use
\begin{align*}
&(W(x)-W(y))(W(x)\xi(x)^2-W(y)\xi(y)^2)\\
&\hspace{1cm}=(Z(x)-Z(y))^2-W(x)W(y)(\xi(x)-\xi(y))^2.
\end{align*}
The corresponding reflected contribution is bounded below by \((Z(x)^2+Z(y)^2)|x-y^\lambda|^{-N-2s}\), because its remaining terms are nonnegative. It follows that
\begin{align*}
\frac{2}{C_{N,s}}\E(w,B) \ge{}&\iint_{\Sigma\times\Sigma} \frac{(Z(x)-Z(y))^2}{|x-y|^{N+2s}}\dd x\dd y +\iint_{\Sigma\times\Sigma} \frac{Z(x)^2+Z(y)^2}{|x-y^\lambda|^{N+2s}}\dd x\dd y -\mathcal R_\xi.
\end{align*}
Indeed, the original pairing integrands converge in \(L^1\) along \(D_n\), the commutator is integrable by hypothesis, and monotone convergence applies to the retained nonnegative energy terms. Since \(Z=0\) outside \(\Sigma\), changing variables \(z=y^\lambda\) identifies the two nonnegative integrals with \([Z]_{H^s(\R^N)}^2\). This proves \eqref{eq3.2}.
\end{proof}

Fix \(\phi\in C_c^\infty(B_2)\), with \(0\le\phi\le1\) and \(\phi=1\) on \(B_1\), and set \(\phi_R(x)=\phi(x/R)\), \(R>1\). The next lemma specifies the cutoffs used in removing the reflected singular set and the far field.

\begin{lemma}\label{Lem3.4}
Assume condition \ref{Ass1.1}, fix \(\lambda<0\), and let \(W\ge0\) belong to \(L^{2_s^*}(\R^N)\), vanish outside \(\Sigma_\lambda\), and be bounded on a neighborhood of \(\Gamma_\lambda\). There are decreasing bounded open neighborhoods \(U_j\) of \(\Gamma_\lambda\), with
\[
\overline{U_{j+1}}\subset U_j, \qquad \overline{U_j}\Subset\Sigma_\lambda, \qquad \bigcap_j\overline{U_j}=\Gamma_\lambda,
\]
and smooth functions \(\eta_{j,\lambda}\) satisfying
\[
0\le\eta_{j,\lambda}\le1, \qquad \eta_{j,\lambda}=0\text{ near }\Gamma_\lambda, \qquad \eta_{j,\lambda}=1\text{ outside }U_j, \qquad \|1-\eta_{j,\lambda}\|_{H^s}<1/j,
\]
such that, with \(\xi_{j,R}=\eta_{j,\lambda}\phi_R\),
\begin{equation}\label{eq3.3}
\lim_{R\to\infty}\limsup_{j\to\infty} \iint_{\Sigma_\lambda\times\Sigma_\lambda} \frac{W(x)W(y)|\xi_{j,R}(x)-\xi_{j,R}(y)|^2} {|x-y|^{N+2s}}\dd x\dd y=0.
\end{equation}
Moreover, \(\eta_{j,\lambda}\to1\) almost everywhere, and every \(K\Subset\Sigma_\lambda\setminus\Gamma_\lambda\) satisfies \(\xi_{j,R}=1\) on \(K\) for all sufficiently large \(R\) and \(j\).
\end{lemma}

\begin{proof}
Choose a bounded neighborhood \(U_0\Subset\Sigma_\lambda\) of \(\Gamma_\lambda\) on which \(W\) is bounded, and choose the nested \(U_j\) so that \(\overline{U_1}\subset U_0\). Lemmas~\ref{Lem2.1} and \ref{Lem2.3} supply the stated cutoffs. Put \(\zeta_j=1-\eta_{j,\lambda}\) and
\[
J_j(x)=\int_{\R^N} \frac{|\zeta_j(x)-\zeta_j(y)|^2}{|x-y|^{N+2s}}\dd y.
\]
By \(2ab\le a^2+b^2\) and symmetry, the remainder $\mathcal R_{\eta_{j,\lambda}}$ is bounded by \(\int_{\Sigma_\lambda}W^2J_j\). On \(U_0\) this integral is at most \(\|W\|_{L^\infty(U_0)}^2[\zeta_j]_{H^s}^2\), which tends to zero. For \(x\notin U_0\), support separation gives
\[
J_j(x)\le C\|\zeta_j\|_2^2(1+|x|)^{-N-2s}.
\]
The integral of \(W^2(1+|x|)^{-N-2s}\) is finite by H\"older's inequality and \(W\in L^{2_s^*}\). Hence the remainder $\mathcal R_{\eta_{j,\lambda}}$ vanishes as \(j\to\infty\).

For the far-field cutoff define
\[
J_R(x)=\int_{\R^N} \frac{|\phi_R(x)-\phi_R(y)|^2}{|x-y|^{N+2s}}\dd y =R^{-2s}J_1(x/R).
\]
The Lipschitz bound for \(\phi\) makes \(J_1\) bounded, while compact support gives \(J_1(x)\le C|x|^{-N-2s}\) for \(|x|>4\). Consequently \(\sup_{R>1}\|J_R\|_{L^{N/(2s)}}<\infty\), and \(J_R\to0\) uniformly on bounded sets. For fixed \(L>0\),
\[
\int_{\Sigma_\lambda\setminus B_L}W^2J_R \le C\|W\|_{L^{2_s^*}(\R^N\setminus B_L)}^2.
\]
This is uniformly small for large \(L\); on \(B_L\), local \(L^2\)-integrability of \(W\) and uniform convergence of \(J_R\) give convergence to zero. Thus \(\int_{\Sigma_\lambda}W^2J_R\to0\). Finally,
\[
|\eta_{j,\lambda}(x)\phi_R(x)-\eta_{j,\lambda}(y)\phi_R(y)|^2 \le2|\eta_{j,\lambda}(x)-\eta_{j,\lambda}(y)|^2 +2|\phi_R(x)-\phi_R(y)|^2
\]
proves \eqref{eq3.3}. The nested neighborhoods eventually exclude each point outside \(\Gamma_\lambda\), and uniformly exclude every compact subset of its complement. Since \(|\Gamma_\lambda|=0\), this proves both final assertions.
\end{proof}

\section{Estimates for the Reflected Hartree Potential}\label{Sec4}

For $\lambda<0$, put
\[
E_\lambda=\{x\in D_\lambda:w_\lambda(x)>0\}.
\]
Lemma~\ref{Lem2.7} gives $V,V_\lambda\in L^\infty(\R^N)$. Proposition~\ref{Prop2.2} gives continuity of $u$ and $u_\lambda$ in $D_\lambda$, so $E_\lambda$ is open there.

\begin{lemma}\label{Lem4.1}
For every \(x\in D_\lambda\),
\begin{equation}\label{eq4.1}
V(x)-V_\lambda(x) = \int_{\Sigma_\lambda} \left(u^p(y)-u_\lambda^p(y)\right) \left( \frac1{|x-y|^\mu}-\frac1{|x^\lambda-y|^\mu} \right)\dd y.
\end{equation}
The sets of zero capacity \(\Gamma\) and \(\Gamma_\lambda\) may be ignored in the integral.
\end{lemma}

\begin{proof}
Write \(\Sigma=\Sigma_\lambda\). By Lemmas~\ref{Lem2.2} and \ref{Lem2.3}, both \(\Gamma\) and \(\Gamma_\lambda\) have zero Lebesgue measure. Hence, after redefining \(u\) and \(u_\lambda\) arbitrarily on these sets, the Hartree potentials may be computed over all of \(\R^N\).

Fix \(x\in D_\lambda\subset\Sigma\). The hyperplane \(T_\lambda\) has zero Lebesgue measure, and reflection maps
\[
(\R^N\setminus\Sigma)\setminus T_\lambda=\{x_1>\lambda\}
\]
bijectively onto \(\Sigma\). Hence, modulo the null boundary \(T_\lambda\), the change of variables below is valid:
\begin{align*}
V(x) &= \int_{\Sigma}\frac{u^p(y)}{|x-y|^\mu}\dd y + \int_{\R^N\setminus\Sigma}\frac{u^p(y)}{|x-y|^\mu}\dd y \\
&= \int_{\Sigma}\frac{u^p(y)}{|x-y|^\mu}\dd y + \int_{\Sigma}\frac{u^p(z^\lambda)}{|x-z^\lambda|^\mu}\dd z \\
&= \int_{\Sigma}\frac{u^p(y)}{|x-y|^\mu}\dd y + \int_{\Sigma}\frac{u_\lambda^p(z)}{|x^\lambda-z|^\mu}\dd z.
\end{align*}
Renaming \(z\) as \(y\), we get
\begin{equation}\label{eq4.2}
V(x) = \int_{\Sigma}\frac{u^p(y)}{|x-y|^\mu}\dd y + \int_{\Sigma}\frac{u_\lambda^p(y)}{|x^\lambda-y|^\mu}\dd y.
\end{equation}
Applying the same decomposition to \(V_\lambda(x)=V(x^\lambda)\), we find
\begin{equation}\label{eq4.3}
V_\lambda(x) = \int_{\Sigma}\frac{u_\lambda^p(y)}{|x-y|^\mu}\dd y + \int_{\Sigma}\frac{u^p(y)}{|x^\lambda-y|^\mu}\dd y.
\end{equation}
Subtracting \eqref{eq4.3} from \eqref{eq4.2} yields \eqref{eq4.1}.
\end{proof}

\begin{lemma}\label{Lem4.2}
For every \(x,y\in\overline{\Sigma_\lambda}\) with \(x\ne y\),
\[
\frac1{|x-y|^\mu}-\frac1{|x^\lambda-y|^\mu}\ge0.
\]
The inequality is strict whenever \(x,y\in\Sigma_\lambda\).
\end{lemma}

\begin{proof}
A direct computation gives
\[
|x^\lambda-y|^2-|x-y|^2 =4(\lambda-x_1)(\lambda-y_1)\ge0.
\]
Because \(x\ne y\), the kernels are finite. Since \(r\mapsto r^{-\mu}\) is strictly decreasing, the asserted inequality follows. If \(x,y\in\Sigma_\lambda\), both factors on the right are positive, and the inequality is strict.
\end{proof}

\begin{lemma}\label{Lem4.3}
For every \(x\in D_\lambda\),
\begin{equation}\label{eq4.4}
\bigl(V(x)-V_\lambda(x)\bigr)^+ \le p\int_{E_\lambda} \frac{u^{p-1}(y)W_\lambda(y)}{|x-y|^\mu}\dd y.
\end{equation}
\end{lemma}

\begin{proof}
Fix \(x\in D_\lambda\). By Lemma~\ref{Lem4.1},
\[
V(x)-V_\lambda(x) = \int_{\Sigma_\lambda} \bigl(u^p(y)-u_\lambda^p(y)\bigr)K_\lambda(x,y)\dd y,
\]
where
\[
K_\lambda(x,y) := \frac1{|x-y|^\mu}-\frac1{|x^\lambda-y|^\mu}.
\]
The kernel is understood off the diagonal \(y=x\), which is negligible in the integral. By Lemma~\ref{Lem4.2},
\[
0\le K_\lambda(x,y)\le |x-y|^{-\mu} \quad\text{for a.e. }y\in\Sigma_\lambda.
\]
Moreover, the integrand in \eqref{eq4.1} is absolutely integrable:
\[
\int_{\Sigma_\lambda} |u^p(y)-u_\lambda^p(y)|K_\lambda(x,y)\dd y \le \int_{\Sigma_\lambda} \frac{u^p(y)+u_\lambda^p(y)}{|x-y|^\mu}\dd y <+\infty.
\]
The last two integrals are finite by the decompositions \eqref{eq4.2}-\eqref{eq4.3}. Consequently the positive and negative parts may be separated. On \(D_\lambda\setminus E_\lambda\) one has \(u\le u_\lambda\), hence
\[
\bigl(u^p(y)-u_\lambda^p(y)\bigr)K_\lambda(x,y)\le0 \qquad\text{for a.e. }y\in D_\lambda\setminus E_\lambda.
\]
Therefore,
\[
\bigl(V(x)-V_\lambda(x)\bigr)^+ \le \int_{E_\lambda} \bigl(u^p(y)-u_\lambda^p(y)\bigr)K_\lambda(x,y)\dd y.
\]
Since \(p>1\) and \(u_\lambda\le u\) on \(E_\lambda\), the mean value theorem gives
\[a^p-b^p\le p\,a^{p-1}(a-b)
\qquad\text{whenever }0\le b\le a.
\]
Applying this with \(a=u(y)\) and \(b=u_\lambda(y)\), and using \(u-u_\lambda=W_\lambda\) on \(E_\lambda\), we get
\[u^p(y)-u_\lambda^p(y)
\le p\,u^{p-1}(y)W_\lambda(y).
\]
Finally, since \(K_\lambda(x,y)\le |x-y|^{-\mu}\), we conclude that
\[
\bigl(V(x)-V_\lambda(x)\bigr)^+ \le p\int_{E_\lambda} \frac{u^{p-1}(y)W_\lambda(y)}{|x-y|^\mu}\dd y,
\]
which proves \eqref{eq4.4}.
\end{proof}

\begin{proposition}\label{Prop4.1}
Let \(\lambda<0\). Under conditions \ref{Ass1.1} and \ref{Ass1.2}, one has \(W_\lambda\in\dot H^s(\R^N)\) and
\begin{equation}\label{eq4.5}
[W_\lambda]_{H^s(\R^N)}^2 \le C\mathcal M(E_\lambda)[W_\lambda]_{H^s(\R^N)}^2,
\end{equation}
where \(\mathcal M\) is defined by \eqref{eq2.3}. The constant \(C\) is independent of \(\lambda\) and \(E_\lambda\).
\end{proposition}

\begin{proof}
Lemma~\ref{Lem2.7} gives
\[
W_\lambda\in L^{2_s^*}(\R^N), \qquad 0\le W_\lambda\le u\quad\text{in }\Sigma_\lambda.
\]
Moreover, \(W_\lambda\) is bounded near \(\Gamma_\lambda\), since this set is separated from \(\Gamma\). Use the cutoffs \(\xi_{j,R}\) from Lemma~\ref{Lem3.4}, and put
\[
\Psi_{j,R}=W_\lambda\xi_{j,R}^2, \qquad Z_{j,R}=W_\lambda\xi_{j,R}.
\]
Write \(\mathcal R_{j,R}=\mathcal R_{\xi_{j,R}}\) for the remainder in Lemma~\ref{Lem3.3}. It is finite by the estimates in Lemma~\ref{Lem3.4}. Lemma~\ref{Lem3.2} gives the absolutely convergent identity
\[
\E(w_\lambda,\Psi_{j,R}) =\int_{E_\lambda}(Vu^q-V_\lambda u_\lambda^q)W_\lambda\xi_{j,R}^2\dd x.
\]
On \(E_\lambda\), one has
\[
u^q-u_\lambda^q\le q u^{q-1}W_\lambda, \qquad u_\lambda^q\le u^q.
\]
Since \(0\le\xi_{j,R}\le1\), Lemma~\ref{Lem4.3} therefore yields
\begin{align*}
\E(w_\lambda,\Psi_{j,R}) \le q\int_{E_\lambda}Vu^{q-1}W_\lambda^2\dd x +p\iint_{E_\lambda\times E_\lambda} \frac{u^{p-1}(y)W_\lambda(y)u^q(x)W_\lambda(x)}{|x-y|^\mu}\dd y\dd x.
\end{align*}
Both terms are finite by the \(L^{2_s^*}\)-integrability of \(W_\lambda\) and the stated \(L^{m_H}\)-bounds for \(u\). Indeed, the HLS and H\"older steps in the proof of Proposition~\ref{Prop2.1}, before applying the Sobolev inequality, give
\begin{align*}
\int_{E_\lambda}Vu^{q-1}W_\lambda^2\dd x &\le C\|u\|_{L^{m_H}(\R^N)}^p \|u\|_{L^{m_H}(E_\lambda)}^{q-1} \|W_\lambda\|_{L^{2_s^*}(\R^N)}^2,\\
\iint_{E_\lambda\times E_\lambda} \frac{u^{p-1}(y)W_\lambda(y)u^q(x)W_\lambda(x)}{|x-y|^\mu}\dd y\dd x &\le C\|u\|_{L^{m_H}(E_\lambda)}^{p+q-1} \|W_\lambda\|_{L^{2_s^*}(\R^N)}^2.
\end{align*}
Combining these bounds with Lemma~\ref{Lem3.3}, we obtain
\[
[Z_{j,R}]_{H^s(\R^N)}^2 \le C\mathcal M(E_\lambda)\|W_\lambda\|_{L^{2_s^*}(\R^N)}^2 +\mathcal R_{j,R}.
\]
By Lemma~\ref{Lem3.4}, there are \(R_n\to\infty\) and \(j_n\to\infty\) such that
\[
\mathcal R_{j_n,R_n}\to0, \qquad \xi_{j_n,R_n}\to1\quad\text{a.e. in }\R^N.
\]
Fatou's lemma in the Gagliardo integral gives
\[
[W_\lambda]_{H^s(\R^N)}^2 \le C\mathcal M(E_\lambda)\|W_\lambda\|_{L^{2_s^*}(\R^N)}^2<\infty.
\]
Together with \(W_\lambda\in L^{2_s^*}(\R^N)\), this proves \(W_\lambda\in\dot H^s(\R^N)\). The homogeneous Sobolev inequality now yields \eqref{eq4.5}.
\end{proof}

\begin{corollary}\label{Cor4.1}
If
\begin{equation}\label{eq4.6}
C\mathcal M(E_\lambda)<1,
\end{equation}
then
\[
W_\lambda\equiv0\quad\text{in }\R^N.
\]
Equivalently,
\[u(x)\le u_\lambda(x)
\quad\text{for a.e. }x\in\Sigma_\lambda.
\]
By Proposition~\ref{Prop2.2}, the inequality holds pointwise in \(D_\lambda\).
\end{corollary}

\begin{proof}
By Proposition~\ref{Prop4.1} and \eqref{eq4.6}, $[W_\lambda]_{H^s}=0$. The Sobolev inequality gives $W_\lambda=0$ almost everywhere, and continuity in $D_\lambda$ gives the pointwise comparison.
\end{proof}

\begin{corollary}\label{Cor4.2}
Fix a constant \(C>0\) for which Proposition~\ref{Prop4.1} holds, and define
\[
\Theta(t)=t^{p+q-1} +\|u\|_{L^{m_H}(\R^N)}^p t^{q-1}, \qquad t\ge0.
\]
Then \(\Theta\) is continuous and strictly increasing, \(\Theta(0)=0\), and \(\Theta(t)\to+\infty\) as \(t\to+\infty\). Consequently there exists a unique number
\[
\tau_*=\Theta^{-1}(C^{-1})>0, \qquad\text{equivalently}\qquad C\Theta(\tau_*)=1.
\]
For every \(\lambda<0\),
\[
\|u\|_{L^{m_H}(E_\lambda)}<\tau_* \quad\Longrightarrow\quad W_\lambda\equiv0 \quad\Longrightarrow\quad u\le u_\lambda\text{ in }D_\lambda.
\]
The threshold \(\tau_*\) may depend on the fixed solution \(u\) and on the structural parameters and is independent of the plane position \(\lambda\).
\end{corollary}

\begin{proof}
By \eqref{eq2.3},
\[
\mathcal M(E_\lambda) =\Theta\!\left(\|u\|_{L^{m_H}(E_\lambda)}\right).
\]
Thus the stated strict inequality implies \(C\mathcal M(E_\lambda)<1\), and Corollary~\ref{Cor4.1} applies.
\end{proof}

\section{Strong Comparison and Moving Planes}\label{Sec5}

In this section, $u$ satisfies the hypotheses of Theorem~\ref{Thm1.1}.
We shall use the following strong maximum principle for antisymmetric supersolutions. For a bounded open set \(\Omega\subset\R^N\), write
\[
\mathcal D^s(\Omega):= \left\{ v\in \Ls(\R^N):\ \text{there exists an open set }U\text{ with }\overline\Omega\subset U\Subset\R^N \text{ such that }v\in H^s(U) \right\}.
\]
For functions in this class, the local and tail estimates in Lemma~\ref{Lem3.2} show that $\E(v,\varphi)$ is absolutely convergent and continuous in $\varphi\in H_0^s(\Omega)$. This class is therefore contained in the class used by Fall and Jarohs \cite{FallJarohs}, which is defined by finiteness of these pairings. Also $cv\in L^2(\Omega)$ for $c\in L^\infty(\Omega)$, as required in their weak formulation.

\begin{lemma}\label{Lem5.1}
Let \(H\subset\R^N\) be a half-space, let \(\Omega\subset H\) be bounded, connected, and open, and let \(c\in L^\infty(\Omega)\). Assume that \(v\in \mathcal D^s(\Omega)\) is an entire antisymmetric supersolution of
\[
(-\Delta)^s v=c(x)v \quad\text{in }\Omega,
\]
that is,
\[
\E(v,\varphi)\ge \int_\Omega c(x)v\varphi\dd x
\]
for every nonnegative \(\varphi\in H_0^s(\Omega)\), \(v\ge0\) a.e. in \(H\), and \(v(x^\ast)=-v(x)\) for a.e. \(x\in H\), where \(x^\ast\) denotes reflection across \(\partial H\). Then either \(v\equiv0\) a.e. in \(\R^N\), or
\[
\operatorname*{ess\,inf}_{K} v>0 \quad\text{for every compact set }K\Subset\Omega.
\]
In particular, if \(v\) has a continuous representative in \(\Omega\), then in the second case \(v>0\) everywhere in \(\Omega\).
\end{lemma}

\begin{proof}
The preceding pairing estimate verifies the function-space hypotheses of Corollary~3.4 of Fall and Jarohs \cite{FallJarohs}. Its vanishing alternative holds in $H$; antisymmetry then gives vanishing almost everywhere in $\R^N$.

We detail the compact-set lower bound. Suppose that $v$ is not zero almost everywhere in $H$, and fix $x_0\in\Omega$. Since $|\{x\in H:v(x)>0\}|>0$, we can choose $r>0$ so small that $B_r(x_0)\Subset\Omega$, $\lambda_1(B_r(x_0))>\|c\|_{L^\infty(\Omega)}$, and the positive set of $v$ outside $\overline{B_r(x_0)}$ has positive measure. Here $\lambda_1$ denotes the first Dirichlet eigenvalue of $(-\Delta)^s$. Hence there are $\delta>0$ and a compact set $A\Subset H\setminus\overline{B_r(x_0)}$ of positive measure such that $v\ge\delta$ almost everywhere on $A$. Proposition~3.3 of \cite{FallJarohs}, applied on $B_r(x_0)$, gives
\[
v(x)\ge d\,\dist(x,\R^N\setminus B_r(x_0))^s \quad\text{for a.e. }x\in B_r(x_0)
\]
for some $d>0$. In particular, $v$ has a positive essential lower bound on $B_{r/2}(x_0)$. A finite covering proves the stated bound on every compact subset of $\Omega$. Continuity gives the pointwise conclusion.
\end{proof}

\begin{samepage}
\begin{proposition}\label{Prop5.1}
Let \(\lambda<0\) and suppose that
\[
u\le u_\lambda\quad\text{in }D_\lambda.
\]
Then exactly one of the following alternatives holds:
\[
w_\lambda\equiv0\quad\text{a.e. in }\R^N,
\]
or
\[
u(x)<u_\lambda(x) \quad\text{for every }x\in D_\lambda.
\]
In the first alternative, continuity gives \(u\equiv u_\lambda\) pointwise in \(D_\lambda\). The analogous alternative holds for \(\lambda>0\) in
\[
D_\lambda^+=\{x_1>\lambda\}\setminus(\Gamma\cup\Gamma_\lambda).
\]
\end{proposition}
\end{samepage}

\begin{proof}
The function $-w_\lambda=u_\lambda-u$ is antisymmetric with respect to $T_\lambda$ and is nonnegative almost everywhere in $\Sigma_\lambda$. Fix a ball $B\Subset D_\lambda$. Local $H^s$ regularity and the $\Ls$ condition imply $-w_\lambda\in\mathcal D^s(B)$. By Lemmas~\ref{Lem4.1} and \ref{Lem4.2},
\[
V_\lambda(x)-V(x) =\int_{\Sigma_\lambda}(u_\lambda^p(y)-u^p(y)) \left(\frac1{|x-y|^\mu}-\frac1{|x^\lambda-y|^\mu}\right)\dd y \ge0\qquad(x\in B).
\]
Since $u_\lambda\ge u$ on the comparison half-space,
\[
\E(-w_\lambda,\varphi) =\int_B(V_\lambda u_\lambda^q-Vu^q)\varphi\dd x\ge0
\]
for nonnegative $\varphi\in C_c^\infty(B)$. The local testing argument in Lemma~\ref{Lem3.2} extends this inequality to every nonnegative $\varphi\in H_0^s(B)$. Lemma~\ref{Lem5.1}, with $c=0$, implies either $w_\lambda=0$ almost everywhere in $\R^N$ or $-w_\lambda>0$ in $B$. If the global identity does not hold, choose such a ball about each point of $D_\lambda$ and use continuity. This proves strict comparison throughout $D_\lambda$. Reflection of the first coordinate gives the right half-space assertion.
\end{proof}

\begin{lemma}\label{Lem5.2}
There exists \(M>0\) such that for every \(\lambda<-M\),
\begin{equation}\label{eq5.1}
u(x)\le u_\lambda(x) \quad\text{for every }x\in D_\lambda.
\end{equation}
\end{lemma}

\begin{proof}
By the tail convergence used in Lemma~\ref{Lem2.6}(ii),
\[
\|u\|_{L^{m_H}(\Sigma_\lambda)}\to0 \quad\text{as }\lambda\to-\infty.
\]
Choose \(M>0\) so large that
\[
\|u\|_{L^{m_H}(\Sigma_\lambda)}<\tau_* \quad\text{for every }\lambda<-M,
\]
where \(\tau_*\) is defined in Corollary~\ref{Cor4.2}. Since \(E_\lambda\subset\Sigma_\lambda\), Corollary~\ref{Cor4.2} yields
\[
u\le u_\lambda \quad\text{in }D_\lambda.
\]
This proves \eqref{eq5.1}.
\end{proof}

Define
\[
\Lambda_-= \left\{ \lambda<0: u\le u_t\text{ in }D_t \text{ for every }t\le\lambda \right\}, \qquad \lambda_0=\sup\Lambda_-.
\]
By Lemma~\ref{Lem5.2}, \(\Lambda_-\ne\emptyset\). The equality \(\lambda_0=0\) will imply that every \(\lambda<0\) belongs to \(\Lambda_-\).

\begin{lemma}\label{Lem5.3}
We have
\begin{equation}\label{eq5.2}
u(x)\le u_{\lambda_0}(x) \quad\text{for every }x\in D_{\lambda_0}.
\end{equation}
\end{lemma}

\begin{proof}
Let \(x\in D_{\lambda_0}\). Since \(x_1<\lambda_0\), we may choose \(\lambda_n\in\Lambda_-\) with \(\lambda_n\uparrow\lambda_0\) and \(x_1<\lambda_n\) for every large \(n\). Moreover, \(x\notin\Gamma_{\lambda_0}\), so
\[
d:=\dist(x,\Gamma_{\lambda_0})>0.
\]
Because
\[
\Gamma_{\lambda_n}=\{(2\lambda_n,\xi'): (0,\xi')\in\Gamma\}
\]
converges to \(\Gamma_{\lambda_0}\) in the Hausdorff metric, we have \(x\notin\Gamma_{\lambda_n}\) for all large \(n\). Hence \(x\in D_{\lambda_n}\) for all large \(n\), and therefore
\[
u(x)\le u_{\lambda_n}(x)=u(x^{\lambda_n}).
\]
Now
\[
x^{\lambda_n}\to x^{\lambda_0} \quad\text{as }n\to\infty,
\]
and
\[
\dist(x^{\lambda_0},\Gamma)=\dist(x,\Gamma_{\lambda_0})=d>0.
\]
Thus the points \(x^{\lambda_n}\) stay in a compact subset of \(\R^N\setminus\Gamma\) for all large \(n\), and Proposition~\ref{Prop2.2} yields the continuity of \(u\) there. Passing to the limit gives
\[
u(x)\le u(x^{\lambda_0})=u_{\lambda_0}(x).
\]
This proves \eqref{eq5.2}.
\end{proof}

\begin{lemma}\label{Lem5.4}
Let \(x_0\in\Gamma\) be a nonremovable singular point. Then
\[
u\not\equiv u_\lambda \quad\text{in }D_\lambda \qquad\text{for every }\lambda<0,
\]
and
\[
u\not\equiv u_\lambda \quad\text{in }D_\lambda^+ \qquad\text{for every }\lambda>0.
\]
\end{lemma}

\begin{proof}
By the definition of a nonremovable singular point and the local H\"older continuity of \(u\), there exists a sequence \(x_n\in\R^N\setminus\Gamma\) such that
\[
x_n\to x_0, \qquad u(x_n)\to+\infty.
\]
Fix first \(\lambda<0\), and set \(y_n=x_n^\lambda\). Then
\[
y_n\to x_0^\lambda, \qquad (x_0^\lambda)_1=2\lambda\ne0,
\]
so \(x_0^\lambda\notin\Gamma\). Moreover, \(y_{n,1}\to2\lambda<\lambda\), and \(y_n\notin\Gamma_\lambda\), since otherwise \(x_n=y_n^\lambda\in\Gamma\). Hence
\[
y_n\in D_\lambda\quad\text{for all large }n.
\]
If \(u\equiv u_\lambda\) in \(D_\lambda\), then
\[
u(y_n)=u_\lambda(y_n)=u(x_n)\to+\infty,
\]
contradicting the local boundedness of the representative near the regular point \(x_0^\lambda\).

For \(\lambda>0\), the same construction gives \(y_{n,1}\to2\lambda>\lambda\), so \(y_n\in D_\lambda^+\) for all large \(n\). The identical contradiction proves the right-hand assertion.
\end{proof}

\begin{lemma}\label{Lem5.5}
For every \(\lambda<0\) for which
\[
u\le u_\lambda\quad\text{in }D_\lambda,
\]
one has
\[
u<u_\lambda\quad\text{in }D_\lambda.
\]
For every \(\lambda>0\) for which
\[
u\le u_\lambda\quad\text{in }D_\lambda^+,
\]
one has
\[
u<u_\lambda\quad\text{in }D_\lambda^+.
\]
\end{lemma}

\begin{proof}
Proposition~\ref{Prop5.1} gives a dichotomy on either side: global identity or strict inequality. Lemma~\ref{Lem5.4} excludes the identity whenever \(\lambda\ne0\). Hence strict inequality holds.
\end{proof}

\begin{lemma}\label{Lem5.6}
Assume \(\lambda_0<0\). There exist \(\delta\in(0,-\lambda_0/2)\), \(R>1\), \(\rho>0\), and \(\eta>0\) such that for all \(\lambda\in(\lambda_0,\lambda_0+\eta)\),
\begin{equation}\label{eq5.3}
E_\lambda \subset (\mathcal S_\delta\cap B_R) \cup (U_\rho(\Gamma_{\lambda_0})\cap B_R) \cup (\R^N\setminus B_R),
\end{equation}
where
\[
\mathcal S_\delta=\{x\in\R^N:|x_1-\lambda_0|<\delta\}, \qquad U_\rho(\Gamma_{\lambda_0})=\{x:\dist(x,\Gamma_{\lambda_0})<\rho\}.
\]
Moreover,
\begin{equation}\label{eq5.4}
\mathcal M(E_\lambda)<\frac1C,
\end{equation}
where \(C\) is the constant in Proposition~\ref{Prop4.1}.
\end{lemma}

\begin{proof}
Fix \(0<\varepsilon<\tau_*/2\), where \(\tau_*\) is the uniform threshold from Corollary~\ref{Cor4.2}. By Lemma~\ref{Lem5.3} and the left-hand conclusion of Lemma \ref{Lem5.5},
\[
u<u_{\lambda_0}\quad\text{in }D_{\lambda_0}.
\]
By Lemma~\ref{Lem2.6}(iv), choose \(R>1\) so large that
\begin{equation}\label{eq5.5}
\|u\|_{L^{m_H}(\R^N\setminus B_R)}<\varepsilon.
\end{equation}
Since \(\Gamma_{\lambda_0}\) has zero measure by Lemma~\ref{Lem2.2},
\[
|U_\rho(\Gamma_{\lambda_0})\cap B_R|\to0 \quad\text{as }\rho\downarrow0,
\]
and
\[
|\{x\in B_R: |x_1-\lambda_0|<\delta\}|\to0 \quad\text{as }\delta\downarrow0.
\]
Thus, by the absolute continuity of the Lebesgue integral, we may choose \(\delta\in(0,-\lambda_0/2)\) and \(\rho>0\) so small that
\[
G_{\delta,\rho,R}:= \{x\in B_R: |x_1-\lambda_0|<\delta\} \cup(U_\rho(\Gamma_{\lambda_0})\cap B_R)
\]
satisfies
\begin{equation}\label{eq5.6}
\|u\|_{L^{m_H}(G_{\delta,\rho,R})}<\varepsilon.
\end{equation}

We now prove the localization of \(E_\lambda\). Define the compact set
\[
K_{\delta,\rho,R}:= \overline{B_R}\cap\{x_1\le \lambda_0-\delta\}\setminus U_\rho(\Gamma_{\lambda_0}).
\]
If \(K_{\delta,\rho,R}\ne\emptyset\), then \(K_{\delta,\rho,R}\Subset D_{\lambda_0}\), and Lemma~\ref{Lem5.5} and continuity give
\[
m:= -\max_{K_{\delta,\rho,R}} w_{\lambda_0}>0.
\]
Consider the reflected compact set
\[
K_{\delta,\rho,R}^{\lambda_0}:=\{x^{\lambda_0}:x\in K_{\delta,\rho,R}\}.
\]
Because reflection is an isometry and \(K_{\delta,\rho,R}\cap U_\rho(\Gamma_{\lambda_0})=\emptyset\), one has
\[
\dist\bigl(K_{\delta,\rho,R}^{\lambda_0},\Gamma\bigr)=\dist\bigl(K_{\delta,\rho,R},\Gamma_{\lambda_0}\bigr)\ge \rho.
\]
Hence there exists \(\eta_0>0\) such that the set
\[
\mathcal K:=\{x^\lambda:x\in K_{\delta,\rho,R},\ |\lambda-\lambda_0|\le \eta_0\}
\]
is compactly contained in \(\R^N\setminus\Gamma\). Since \(u\) is continuous on \(\mathcal K\), it is uniformly continuous there. Therefore
\[
\sup_{x\in K_{\delta,\rho,R}}|u(x^\lambda)-u(x^{\lambda_0})|\to0 \quad\text{as }\lambda\to\lambda_0,
\]
and consequently
\[
\sup_{x\in K_{\delta,\rho,R}}|w_\lambda(x)-w_{\lambda_0}(x)|\to0 \quad\text{as }\lambda\to\lambda_0.
\]
Choose \(\eta\in(0,\eta_0)\) so small that \(\eta<\delta/2\) and
\[
\sup_{x\in K_{\delta,\rho,R}}|w_\lambda(x)-w_{\lambda_0}(x)|<\frac m2 \quad\text{for every }\lambda\in(\lambda_0,\lambda_0+\eta).
\]
Then, for such \(\lambda\),
\[
w_\lambda(x)\le -\frac m2<0 \quad\text{for every }x\in K_{\delta,\rho,R}.
\]
Hence
\[
E_\lambda\cap K_{\delta,\rho,R}=\emptyset.
\]
If \(K_{\delta,\rho,R}=\emptyset\), we choose any \(\eta\in(0,\delta/2)\), and the same conclusion is trivial. Now let \(x\in E_\lambda\cap B_R\) with \(x\notin U_\rho(\Gamma_{\lambda_0})\). Since \(x\notin K_{\delta,\rho,R}\), it follows that \(x_1>\lambda_0-\delta\). On the other hand,
\[
x_1<\lambda<\lambda_0+\eta<\lambda_0+\delta.
\]
Therefore
\[
|x_1-\lambda_0|<\delta,
\]
which means that \(x\in \mathcal S_\delta\cap B_R\). This proves \eqref{eq5.3}.

For \(\lambda\in(\lambda_0,\lambda_0+\eta)\), let
\[
F_{\delta,\rho,R}:=G_{\delta,\rho,R}\cup(\R^N\setminus B_R).
\]
By \eqref{eq5.3}, \(E_\lambda\subset F_{\delta,\rho,R}\). Combining \eqref{eq5.5} and \eqref{eq5.6}, we obtain
\[
\|u\|_{L^{m_H}(F_{\delta,\rho,R})}<2\varepsilon.
\]
Therefore, by monotonicity of the \(L^{m_H}\)-norm,
\[
\|u\|_{L^{m_H}(E_\lambda)} \le \|u\|_{L^{m_H}(F_{\delta,\rho,R})} <2\varepsilon<\tau_*.
\]
Corollary~\ref{Cor4.2}, or the identity \(\mathcal M(E_\lambda)=\Theta(\|u\|_{L^{m_H}(E_\lambda)})\), then gives \(\mathcal M(E_\lambda)<1/C\). This proves \eqref{eq5.4}.
\end{proof}

\begin{proposition}\label{Prop5.2}
One has
\begin{equation*}
\lambda_0=0.
\end{equation*}
Consequently,
\begin{equation}\label{eq5.7}
u(x)\le u_\lambda(x) \quad\text{for every }\lambda<0\text{ and every }x\in D_\lambda.
\end{equation}
\end{proposition}

\begin{proof}
Suppose, for a contradiction, that \(\lambda_0<0\). Let \(\eta>0\) be given by Lemma~\ref{Lem5.6}. Then for every \(\lambda\in(\lambda_0,\lambda_0+\eta)\), one has
\[
C\mathcal M(E_\lambda)<1.
\]
Corollary~\ref{Cor4.1} therefore implies
\[
u\le u_\lambda \quad\text{in }D_\lambda
\]
for every \(\lambda\in(\lambda_0,\lambda_0+\eta)\). Since all \(t\le\lambda_0\) already satisfy the comparison by the definition of \(\Lambda_-\), it follows that every
\[
\lambda\in(\lambda_0,\lambda_0+\eta)
\]
belongs to \(\Lambda_-\). This contradicts the definition of \(\lambda_0=\sup\Lambda_-\). Hence \(\lambda_0=0\).

Since \(\Lambda_-\) is downward closed and has supremum \(0\), every \(\lambda<0\) belongs to \(\Lambda_-\). Therefore \eqref{eq5.7}~holds.
\end{proof}

The same arguments give the following.

\begin{proposition}\label{Prop5.3}
We have
\begin{equation}\label{eq5.8}
u(x)\le u_\lambda(x) \quad\text{for every }\lambda>0\text{ and }x\in\{x_1>\lambda\}\setminus(\Gamma\cup\Gamma_\lambda).
\end{equation}
\end{proposition}

\begin{proof}
Set $\widetilde u(x_1,x')=u(-x_1,x')$. The equation, the singular set, and conditions \ref{Ass1.1}-\ref{Ass1.2} are invariant under this reflection; non-removability is preserved as well. Proposition \ref{Prop5.2} therefore applies to $\widetilde u$. For $\lambda>0$ and $x_1>\lambda$, put $z=(-x_1,x')$ and $t=-\lambda<0$. Then $z_1<t$ and
\[
\widetilde u(z)=u(x),\qquad \widetilde u(z^t)=u(x^\lambda).
\]
The condition $x\notin\Gamma\cup\Gamma_\lambda$ is equivalent to $z\notin\Gamma\cup\Gamma_t$. The left comparison for $\widetilde u$ thus gives \eqref{eq5.8}.
\end{proof}

\begin{proof}[Proof of Theorem~\ref{Thm1.1}]
Propositions~\ref{Prop5.2} and \ref{Prop5.3} give the reflected comparisons for every \(\lambda<0\) in \(D_\lambda\) and every \(\lambda>0\) in \(D_\lambda^+\), respectively. For $x_1<0$, let the left plane tend to zero to obtain $u(x_1,x')\le u(-x_1,x')$. Applying the right comparison at $(-x_1,x')$ and letting its plane tend to zero gives the reverse inequality. The reflected points remain away from $\Gamma$ for planes sufficiently close to zero, so both limits follow from continuity. The equality on $\{x_1=0\}\setminus\Gamma$ is immediate. This proves \eqref{eq1.7}.

If $x_1<y_1<0$, choose $\lambda=(x_1+y_1)/2$. Then $(x_1,x')\in D_\lambda$, because its reflection is $(y_1,x')\notin\Gamma$. Proposition~\ref{Prop5.2} yields $u(x_1,x')\le u(y_1,x')$. Symmetry gives the corresponding monotonicity on the positive half-space. Thus \eqref{eq1.8} holds.
\end{proof}

\begin{proof}[Proof of Theorem~\ref{Thm1.2}]
For $x_1<y_1<0$, use the same plane $\lambda=(x_1+y_1)/2$. Proposition~\ref{Prop5.2} and Lemma~\ref{Lem5.5} give
\[
u(x_1,x')<u_\lambda(x_1,x')=u(y_1,x').
\]
Strict decrease on the positive half-space follows from symmetry.
\end{proof}

\section{The Kelvin Transformation}\label{Sec6}

We prove Theorem~\ref{Thm1.3}. Throughout this section, put
\[
d=N-2s,\qquad S=\frac{2N}{d},\qquad P=\frac{2N-\mu}{d},
\]
and assume
\[
N\ge2,\quad 0<s<1,\quad 0<\mu<N,\quad 1<p\le P,\quad 1<q\le P-1, \quad (p,q)\ne(P,P-1).
\]
In particular, $N>2s$, $\mu<4s$, and $p+q+1<2P$.

\begin{lemma}\label{Lem6.1}
Every such weak solution has a locally H\"older continuous representative in $\R^N\setminus\Gamma$. Moreover, $\Hs[u]$ is finite, continuous, and locally bounded there, and the source $\Hs[u]u^q$ is locally bounded.
\end{lemma}

\begin{proof}
Choose $x_*\notin\Gamma$ at which $V(x_*):=\Hs[u](x_*)<\infty$. For every $R>0$,
\[
\int_{B_R}u^p(y)\dd y\le(R+|x_*|)^\mu V(x_*).
\]
Comparison of $|x_*-y|$ with $|y|$ outside a sufficiently large ball therefore gives
\begin{equation}\label{eq6.1}
\int_{\R^N}\frac{u^p(y)}{1+|y|^\mu}\dd y<\infty.
\end{equation}
Fix $U\Subset U'\Subset\R^N\setminus\Gamma$, and choose $\chi\in C_c^\infty(U')$ equal to one near $\overline U$. Local Sobolev embedding gives $u\in L^S(U')$. The decomposition
\[
V(x)=\int_{\R^N}\frac{\chi(y)u^p(y)}{|x-y|^\mu}\dd y +\int_{\R^N}\frac{(1-\chi(y))u^p(y)}{|x-y|^\mu}\dd y
\]
has a smooth bounded second term on $U$. Indeed, its kernel is separated from the diagonal, and all its derivatives on compact subsets of $U$ are dominated by an integrable multiple of the density in \eqref{eq6.1}.

Set $\theta=p/S-(N-\mu)/N$. Since $p\le P<S$, the exponent $S/p$ is greater than one. If $\theta>0$, the Riesz-potential form of the Hardy-Littlewood-Sobolev inequality \cite{LiebLoss} gives $V\in L^{1/\theta}(U)$. If $\theta<0$, H\"older's inequality applied to the local kernel gives $V\in L^\infty(U)$. If $\theta=0$, apply the same Riesz estimate with an exponent slightly below $S/p$; then $V\in L^t(U)$ for every sufficiently large finite $t$. It follows that $c(x):=V(x)u(x)^{q-1}$ belongs locally to $L^k$ for some $k>N/(2s)$. To check this strict inequality when $\theta>0$, observe that
\[
\frac1k=\frac{p+q-1}{S}-\frac{N-\mu}{N}<\frac{2s}{N},
\]
where the last inequality is equivalent to $p+q+1<2P$. When $\theta\le0$, use instead
\[
\frac{q-1}{S}\le\frac{P-2}{S} =\frac{4s-\mu}{2N}<\frac{2s}{N}
\]
and, if necessary, choose $t$ sufficiently large.

We detail the resulting local bootstrap. Fix $k>N/(2s)$ on a larger relatively compact regular open set, and put $\varepsilon_0=2s/N-1/k>0$. If $u\in L^t$ there, initially with $t=S$, then $cu\in L^r$, where $1/r=1/k+1/t$. The initial $r$ is greater than one because $1/k+1/S<(N+2s)/(2N)<1$. On a smaller open set, take a smooth compactly supported cutoff $\zeta$ equal to one near that set and write
\[
u=I_{2s}(\zeta cu)+h.
\]
Here $I_{2s}$ has the normalization for which $(-\Delta)^s I_{2s}f=f$. Both terms defining $h$ belong to $\Ls$; the compactly supported source gives the Riesz potential the decay $O(|x|^{-d})$ at infinity. Thus $h$ is distributionally $s$-harmonic on the smaller set and is smooth there; see \cite{Fall2016}. If $1/r>2s/N$, the Riesz estimate improves the reciprocal exponent of $u$ from $1/t$ to $1/t-\varepsilon_0$. At equality, every finite target exponent is available by slightly lowering $r$ on the bounded support. If $1/r<2s/N$, the potential is locally H\"older continuous, with every exponent $0<\alpha<\min\{2s-N/r,1\}$. Iterating on nested open sets, with a slightly smaller gain if an endpoint is encountered, we reach the last case after finitely many steps. This proves local boundedness and local H\"older continuity of $u$.

Finally, on a compact regular set split $V$ into a small ball about the evaluation point, a bounded part separated from that point, and a far part. Local boundedness of $u$ makes the first contribution uniformly $O(\rho^{N-\mu})$. Dominated convergence applies to the other two parts by \eqref{eq6.1}. This proves continuity and local boundedness of $V$, and hence local boundedness of $Vu^q$. The continuous representative is unique because two continuous functions agreeing almost everywhere agree everywhere.
\end{proof}

Define the nonnegative exponents
\begin{equation}\label{eq6.2}
\tau=2N-\mu-pd,\qquad \sigma=N+2s-\mu-qd.
\end{equation}
They satisfy
\[
\sigma+\tau=d\bigl(2P-p-q-1\bigr)>0.
\]
These signs imply the inequalities between the weights and their reflections used below.

\begin{proposition}\label{Prop6.1}
Let $a\in\{x_1=0\}\setminus\Gamma$, and define
\[
\mathcal I_a(x)=a+\frac{x-a}{|x-a|^2},\qquad v(x)=|x-a|^{-d}u(\mathcal I_a(x)),\qquad K=\mathcal I_a(\Gamma)\cup\{a\}.
\]
Then $K$ is a compact zero-$(s,2)$-capacity subset of $\{x_1=0\}$, and $v\in\Ls(\R^N)\cap H^s_{\loc}(\R^N\setminus K)$ is continuous away from $K$. It satisfies, in the distributional weak sense,
\begin{equation}\label{eq6.3}
(-\Delta)^s v=|x-a|^{-\sigma}\widetilde V(x)v^q, \qquad \widetilde V(x)=\int_{\R^N} \frac{|y-a|^{-\tau}v(y)^p}{|x-y|^\mu}\dd y,
\end{equation}
where
\begin{equation}\label{eq6.4}
\widetilde V(x)=|x-a|^{-\mu}V(\mathcal I_a(x)), \qquad v(x)=O(|x|^{-d})\quad\text{as }|x|\to\infty.
\end{equation}
The potential and source in \eqref{eq6.3} are continuous and locally bounded away from $K$. A nonremovable singular point of $u$ gives a nonremovable singular point of $v$ in $\mathcal I_a(\Gamma)$.
\end{proposition}

\begin{proof}
Choose a bounded open neighborhood $U$ of $\Gamma$ with $\overline U\subset\R^N\setminus\{a\}$. The localized construction in Lemma~\ref{Lem2.1} gives $\psi_j\in C_c^\infty(U)$, with $0\le\psi_j\le1$, $\psi_j=1$ near $\Gamma$, and $\|\psi_j\|_{H^s}\to0$. Put $\widehat\psi_j=\psi_j\circ \mathcal I_a$ away from $a$ and extend it by zero at $a$. These functions are smooth and their supports lie in a fixed bounded annulus $A_0$ about $a$. Choose a slightly larger annulus $A$ with $\overline{A_0}\subset A$. On $A$, the inversion and its inverse are Lipschitz and have bounded Jacobians. Changing variables therefore gives
\[
\|\widehat\psi_j\|_2^2+ \iint_{A\times A} \frac{|\widehat\psi_j(x)-\widehat\psi_j(y)|^2} {|x-y|^{N+2s}}\dd x\dd y \le C\|\psi_j\|_{H^s}^2.
\]
Since $\widehat\psi_j=0$ outside $A_0$ and $\dist(A_0,\R^N\setminus A)>0$, the remaining interactions satisfy
\[
\iint_{A\times(\R^N\setminus A)} \frac{|\widehat\psi_j(x)-\widehat\psi_j(y)|^2} {|x-y|^{N+2s}}\dd x\dd y \le C\|\widehat\psi_j\|_2^2.
\]
Hence $\|\widehat\psi_j\|_{H^s}\to0$. As $\widehat\psi_j=1$ near $\mathcal I_a(\Gamma)$, this proves $\Cap_{s,2}(\mathcal I_a(\Gamma))=0$. A bump of radius $\rho$ about $a$ has squared $H^s$ norm $O(\rho^N+\rho^{N-2s})$, so $\{a\}$ also has zero capacity. The union has zero capacity, is compact, and lies in the same hyperplane because $a_1=0$.

The same change of variables and multiplication by the smooth factor $|x-a|^{-d}$ preserve local $H^s$ regularity away from $K$. Continuity follows from Lemma~\ref{Lem6.1}. Near $a$, the substitution $y=\mathcal I_a(x)$ gives
\[
\int_{B_\rho(a)}v(x)\dd x =\int_{\R^N\setminus B_{1/\rho}(a)} u(y)|y-a|^{-N-2s}\dd y<\infty.
\]
Local integrability on compact sets separated from $a$ follows from that of $u$. At infinity, $\mathcal I_a(x)\to a$, where $u$ is locally bounded, so the decay in \eqref{eq6.4} follows. These observations prove $v\in\Ls(\R^N)$.

The identities
\[
|\mathcal I_a(x)-\mathcal I_a(y)|=\frac{|x-y|}{|x-a||y-a|}, \qquad |\det D\mathcal I_a(y)|=|y-a|^{-2N}
\]
and $\tau+pd=2N-\mu$ give the first identity in \eqref{eq6.4} by a change of variables in a nonnegative integral. Thus the transformed potential is finite and continuous away from $K$. For $\varphi\in C_c^\infty(\R^N\setminus K)$, put $\varphi_a(y)=|y-a|^{-d}\varphi(\mathcal I_a(y))$. Its support is compact and avoids $\Gamma\cup\{a\}$. The fractional Kelvin covariance \cite{ChenLiZhang} reads
\[
(-\Delta)^s\varphi_a(y) =|y-a|^{-N-2s}(( -\Delta)^s\varphi)(\mathcal I_a(y)).
\]
Using this identity in the weak equation for $u$ gives
\begin{align*}
\int_{\R^N}v(-\Delta)^s\varphi\dd x &=\int_{\R^N}u(-\Delta)^s\varphi_a\dd y\\
&=\int_{\R^N}Vu^q\varphi_a\dd y =\int_{\R^N}|x-a|^{-\sigma}\widetilde Vv^q\varphi\dd x.
\end{align*}
The last equality uses $\sigma=N+2s-\mu-qd$. All these pairings are finite by the $\Ls$ condition and local boundedness on the compact supports. The transformed source is locally bounded by \eqref{eq6.4}. If $x_j\to x_0\in\Gamma$ and $u(x_j)\to\infty$, then
\[
v(\mathcal I_a(x_j))=|x_j-a|^d u(x_j)\to\infty, \qquad \mathcal I_a(x_j)\to \mathcal I_a(x_0).
\]
Lemma~\ref{Lem6.1} allows such a sequence to be chosen from the essential-supremum definition of a nonremovable point.
\end{proof}

For the transformed solution, introduce
\begin{equation}\label{eq6.5}
\mathcal A_v=|x-a|^{-\sigma}\widetilde Vv^{q-1},\qquad \mathcal F_v=|x-a|^{-\sigma}v^q,\qquad \mathcal G_v=|x-a|^{-\tau}v^{p-1}.
\end{equation}
Also put
\[
b=\frac{2N}{N+2s-\mu},\qquad a_0=\frac{2N}{2N-\mu},\qquad \mathcal N_v(E)= \|\mathcal A_v\|_{L^{N/(2s)}(E)} +\|\mathcal F_v\|_{L^b(E)}\|\mathcal G_v\|_{L^b(E)}.
\]

\begin{lemma}\label{Lem6.2}
For every $\delta>0$, with $U_\delta(K)=\{x:\dist(x,K)<\delta\}$,
\[
\mathcal A_v\in L^{N/(2s)}(\R^N\setminus U_\delta(K)), \qquad \mathcal F_v,\mathcal G_v\in L^b(\R^N\setminus U_\delta(K)).
\]
Moreover, $v\in L^S$ and $\mathcal F_vv,\mathcal G_vv\in L^{a_0}$ on the same set. Each of these norms tends to zero when restricted to $\R^N\setminus B_R$ as $R\to\infty$. If $Z\in L^S(\R^N)$ is nonnegative and supported in a measurable set $E$ on which the indicated norms are finite, then
\begin{equation}\label{eq6.6}
\int_E\mathcal A_vZ^2\dd x+ \iint_{E\times E} \frac{\mathcal F_v(x)Z(x)\mathcal G_v(y)Z(y)}{|x-y|^\mu}\dd x\dd y \le C\mathcal N_v(E)\|Z\|_{L^S(\R^N)}^2.
\end{equation}
\end{lemma}

\begin{proof}
All the functions under consideration are locally bounded away from $K$. Using \eqref{eq6.2}-\eqref{eq6.5} and boundedness of $u,V$ near $a$, we obtain
\[
\mathcal A_v(x)=|x-a|^{-4s} V(\mathcal I_a(x))u(\mathcal I_a(x))^{q-1}=O(|x|^{-4s}),
\]
\[
\mathcal F_v(x),\mathcal G_v(x) =O(|x|^{-(N+2s-\mu)}),\qquad \mathcal F_v(x)v(x),\mathcal G_v(x)v(x) =O(|x|^{-(2N-\mu)}).
\]
These bounds and $v=O(|x|^{-d})$ prove all the integrability assertions; raising each bound to its stated exponent gives $O(|x|^{-2N})$. The tail limits follow by absolute continuity of the integrals. Finally,
\[
\frac2{a_0}+\frac\mu N=2,\qquad \frac1{a_0}=\frac1b+\frac1S, \qquad \frac{2s}{N}+\frac2S=1.
\]
Apply H\"older's inequality to the first term of \eqref{eq6.6}, and HLS with exponents $a_0,a_0$, followed by H\"older, to its second term. All exponents lie strictly between one and infinity.
\end{proof}

\begin{proposition}\label{Prop6.2}
There exists $\kappa>0$, depending only on $N,s,\mu,p,q$, with the following property. For $\lambda<0$, let
\[
K_\lambda=R_\lambda K,\quad v_\lambda(x)=v(x^\lambda),\quad E_\lambda=\{x\in\Sigma_\lambda\setminus K_\lambda:v(x)>v_\lambda(x)\},\quad W=(v-v_\lambda)^+\one_{\Sigma_\lambda}.
\]
We set $W=0$ on $K_\lambda$. All integrals below are invariant under changes on this null set. Then
\begin{equation}\label{eq6.7}
\mathcal N_v(E_\lambda)<\kappa \quad\Longrightarrow\quad v\le v_\lambda\quad\text{in }\Sigma_\lambda\setminus K_\lambda.
\end{equation}
\end{proposition}

\begin{proof}
Since $a_1=0$ and $\lambda<0$, for $x\in\Sigma_\lambda\setminus K_\lambda$,
\[
|x-a|>|x^\lambda-a|, \qquad |x-a|^{-\sigma}\le|x^\lambda-a|^{-\sigma}, \qquad |x-a|^{-\tau}\le|x^\lambda-a|^{-\tau}.
\]
Put $h(x)=|x-a|^{-\tau}v(x)^p$. Splitting the potential into two half-spaces and reflecting gives
\[
\widetilde V(x)-\widetilde V(x^\lambda) =\int_{\Sigma_\lambda}(h(y)-h(y^\lambda)) \left(\frac1{|x-y|^\mu}-\frac1{|x^\lambda-y|^\mu}\right)\dd y.
\]
The identity is absolutely convergent at regular points, since its absolute integrand is controlled by terms in the two finite potentials. The kernel difference is nonnegative. The inequalities between the reflected weights and $p>1$ imply
\begin{equation}\label{eq6.8}
(\widetilde V-\widetilde V_\lambda)^+(x) \le p\int_{E_\lambda} \frac{\mathcal G_v(y)W(y)}{|x-y|^\mu}\dd y.
\end{equation}
Indeed, $h-h_\lambda\le0$ outside $E_\lambda$, while on $E_\lambda$ it is bounded above by $|y-a|^{-\tau}(v^p-v_\lambda^p)\le p\mathcal G_vW$. Set $f=|x-a|^{-\sigma}\widetilde Vv^q$. On $E_\lambda$, the inequality between $|x-a|^{-\sigma}$ and $|x^\lambda-a|^{-\sigma}$ gives
\begin{equation}\label{eq6.9}
f-f_\lambda\le q\mathcal A_vW +\mathcal F_v(\widetilde V-\widetilde V_\lambda)^+.
\end{equation}

Choose the capacitary and far-field cutoffs from Section~\ref{Sec3}, with $0\le\eta_j,\phi_R\le1$, $\eta_j$ vanishing near $K_\lambda$, and $\eta_j=1$ outside shrinking neighborhoods of $K_\lambda$. Set
\[
\xi_{j,R}=\eta_j\phi_R,\qquad Z_{j,R}=W\xi_{j,R}.
\]
Lemma~\ref{Lem3.2} applies: the two functions are locally $H^s$ and continuous away from their singular sets, belong to $\Ls$, and have locally bounded sources. Consequently $Z_{j,R}\in H^s(\R^N)$, and the reflected weak equation can be tested with $W\xi_{j,R}^2$. Lemma~\ref{Lem3.3} yields
\begin{equation}\label{eq6.10}
\E(v-v_\lambda,W\xi_{j,R}^2) \ge c_s[Z_{j,R}]_{H^s}^2-C_s\mathcal R_{j,R},
\end{equation}
where
\[
\mathcal R_{j,R}= \iint_{\Sigma_\lambda\times\Sigma_\lambda} \frac{W(x)W(y)|\xi_{j,R}(x)-\xi_{j,R}(y)|^2} {|x-y|^{N+2s}}\dd x\dd y.
\]
Since $0\le W\le v$ on $\Sigma_\lambda$, Lemma~\ref{Lem6.2} gives $W\in L^S(\Sigma_\lambda)$; moreover, $W$ is bounded near $K_\lambda$, which has positive distance from $K$. Lemma~\ref{Lem3.4} gives
\begin{equation}\label{eq6.11}
\lim_{R\to\infty}\limsup_{j\to\infty}\mathcal R_{j,R}=0.
\end{equation}

Combining \eqref{eq6.8} and \eqref{eq6.9} with the tested equation, and using $0\le\xi_{j,R}\le1$, gives
\begin{equation}\label{eq6.12}
\begin{aligned}
\E(v-v_\lambda,W\xi_{j,R}^2) &\le q\int_{E_\lambda}\mathcal A_vW^2\dd x +p\iint_{E_\lambda\times E_\lambda} \frac{\mathcal F_v(x)W(x)\mathcal G_v(y)W(y)} {|x-y|^\mu}\dd x\dd y\\
&\le C\mathcal N_v(E_\lambda) \|W\|_{L^S(\R^N)}^2.
\end{aligned}
\end{equation}
The last bound follows from \eqref{eq6.6}. In particular, both terms are finite by the already established $L^S$-integrability of $W$. Combining \eqref{eq6.10} and \eqref{eq6.12}, we obtain
\[
[Z_{j,R}]_{H^s}^2 \le C\mathcal N_v(E_\lambda)\|W\|_{L^S(\R^N)}^2 +C\mathcal R_{j,R}.
\]
By \eqref{eq6.11}, choose $R_n\to\infty$ and $j_n\to\infty$ so that $\mathcal R_{j_n,R_n}\to0$ and $\xi_{j_n,R_n}\to1$ almost everywhere. Fatou's lemma in the Gagliardo integral yields
\begin{equation}\label{eq6.13}
[W]_{H^s}^2 \le C\mathcal N_v(E_\lambda)\|W\|_{L^S(\R^N)}^2<\infty.
\end{equation}
By \eqref{eq6.13}, $W\in\dot H^s(\R^N)$, and the homogeneous Sobolev inequality now gives
\[
[W]_{H^s}^2 \le C\mathcal N_v(E_\lambda)[W]_{H^s}^2.
\]
Choose $\kappa>0$ so that $C\kappa<1$. Under the hypothesis in \eqref{eq6.7}, we obtain $[W]_{H^s}=0$, and hence $W=0$ almost everywhere because $W\in L^S(\R^N)$. Continuity away from $K_\lambda$ proves \eqref{eq6.7}.
\end{proof}

\begin{proposition}\label{Prop6.3}
For every regular center $a$ in the singular hyperplane, the corresponding Kelvin transform $v$ is symmetric with respect to that hyperplane, strictly increasing in $x_1$ on $\{x_1<0\}$, and strictly decreasing in $x_1$ on $\{x_1>0\}$.
\end{proposition}

\begin{proof}
First suppose that $v\le v_\lambda$ in $\Sigma_\lambda\setminus K_\lambda$, where $\lambda<0$. Put $h=|x-a|^{-\tau}v^p$ and $f=|x-a|^{-\sigma}\widetilde Vv^q$. The inequalities $v\le v_\lambda$ and $|x-a|^{-\tau}\le|x^\lambda-a|^{-\tau}$ give $h\le h_\lambda$. The reflection formula in the proof of Proposition~\ref{Prop6.2} then gives $\widetilde V\le\widetilde V_\lambda$. Together with $|x-a|^{-\sigma}\le|x^\lambda-a|^{-\sigma}$, this yields $f\le f_\lambda$. On every ball $B\Subset\Sigma_\lambda\setminus K_\lambda$, $v_\lambda-v$ is an antisymmetric supersolution of $(-\Delta)^s(v_\lambda-v)\ge0$ and belongs to $\mathcal D^s(B)$. The weak inequality extends to nonnegative $H_0^s(B)$ tests by the local testing argument in Lemma~\ref{Lem3.2}. Lemma~\ref{Lem5.1}, with $c=0$, gives strict positivity on $B$ unless $v=v_\lambda$ almost everywhere in all of $\R^N$. The case $v\equiv v_\lambda$ is excluded by the singularity. Indeed, choose $x_j$ approaching a nonremovable point of $K$ with $v(x_j)\to\infty$. The reflected points $x_j^\lambda$ approach a regular point with first coordinate $2\lambda$. Continuity would give $v(x_j^\lambda)=v(x_j)\to\infty$, contradicting local boundedness at that regular point. Thus $v\le v_\lambda$ implies $v<v_\lambda$ throughout $\Sigma_\lambda\setminus K_\lambda$.

By Lemma~\ref{Lem6.2}, $\mathcal N_v(\Sigma_\lambda)\to0$ as $\lambda\to-\infty$. Proposition~\ref{Prop6.2} gives $v\le v_\lambda$ for all sufficiently negative $\lambda$. Define
\[
\Lambda=\{\lambda<0:v\le v_t\text{ in }\Sigma_t\setminus K_t \text{ for every }t\le\lambda\},\qquad \lambda_0=\sup\Lambda.
\]
The set is nonempty and downward closed. Continuity gives comparison at $\lambda_0$ if $\lambda_0<0$: for a fixed regular comparison point, its reflected images remain regular for planes sufficiently close to $\lambda_0$, since the reflected singular sets converge in the Hausdorff metric. The preceding strong comparison then gives $v<v_{\lambda_0}$ on $\Sigma_{\lambda_0}\setminus K_{\lambda_0}$.

Suppose $\lambda_0<0$. Choose $\varepsilon>0$ such that $3\varepsilon+9\varepsilon^2<\kappa$. Choose $R$ large enough that each of the three norms of $\mathcal A_v,\mathcal F_v,\mathcal G_v$ in the definition of $\mathcal N_v$ is less than $\varepsilon$ outside $B_R$. Choose small $\delta\in(0,-\lambda_0/2)$ and $\rho>0$ so that each of those norms is also less than $\varepsilon$ on each of
\[
B_R\cap\{|x_1-\lambda_0|<\delta\},\qquad B_R\cap U_\rho(K_{\lambda_0}).
\]
This is possible by absolute continuity: these bounded sets lie in a fixed region separated from $K$, and their measures tend to zero. On the compact set
\[
L=\overline B_R\cap\{x_1\le\lambda_0-\delta\} \setminus U_\rho(K_{\lambda_0}),
\]
strict comparison and uniform continuity imply $v<v_\lambda$ for all $\lambda\in(\lambda_0,\lambda_0+\eta)$, with some $\eta\in(0,\delta/2)$. If $L$ is empty, take any $\eta\in(0,\delta/2)$. Consequently, for those planes,
\[
E_\lambda\subset (\R^N\setminus B_R)\cup \bigl(B_R\cap\{|x_1-\lambda_0|<\delta\}\bigr) \cup\bigl(B_R\cap U_\rho(K_{\lambda_0})\bigr).
\]
Thus $\mathcal N_v(E_\lambda)<3\varepsilon+9\varepsilon^2<\kappa$. Proposition~\ref{Prop6.2} extends the comparison beyond $\lambda_0$, contradicting its definition. Therefore $\lambda_0=0$, and
\begin{equation}\label{eq6.14}
v(x)<v(x^\lambda) \quad(\lambda<0,\ x\in\Sigma_\lambda\setminus K_\lambda).
\end{equation}
To justify the assertion for every $\lambda<0$, choose an element of $\Lambda$ between $\lambda$ and zero and use downward closure; strictness is the strong comparison proved above.

Reflection across $\{x_1=0\}$ preserves $|x-a|$, $K$, and the weighted equation. Applying the same argument to $v(-x_1,x')$ gives the corresponding comparison from the right. Letting the two planes tend to zero and using continuity proves $v(x_1,x')=v(-x_1,x')$ away from $K$. For $r<t<0$, use \eqref{eq6.14} with $\lambda=(r+t)/2$. The point $(r,x')$ does not belong to $K_\lambda$, whose first coordinate is $r+t$. This gives $v(r,x')<v(t,x')$. Strict decrease on the right follows from symmetry.
\end{proof}

\begin{lemma}\label{Lem6.3}
Suppose that all the Kelvin transforms with centers in $\{x_1=0\}\setminus\Gamma$ have the conclusions of Proposition~\ref{Prop6.3}. Then $u$ is symmetric with respect to $\{x_1=0\}$ and
\begin{equation}\label{eq6.15}
u(r,x')\le u(t,x')\qquad(r<t<0).
\end{equation}
\end{lemma}

\begin{proof}
Let $\mathscr R(x_1,x')=(-x_1,x')$. Since $a_1=0$, $\mathcal I_a\mathscr R=\mathscr R \mathcal I_a$ and $|\mathscr R x-a|=|x-a|$. The symmetry of any one Kelvin transform therefore gives $u(y)=u(\mathscr R y)$ for $y\notin\Gamma\cup\{a\}$; at $a$ the identity is~automatic.

For monotonicity, fix $r<t<0$ and $x'\in\R^{N-1}$. Choose a unit vector $e\in\R^{N-1}$ and, for large integers $k$, set $a_k=(0,x'+ke)$. Compactness of $\Gamma$ ensures $a_k\notin\Gamma$. Put $D_k=k^2+r^2$ and $H_k=k^2+t^2$, and set
\[
X_k=\left(\frac r{D_k},x'+ke-\frac{ke}{D_k}\right),\qquad Y_k=\left(\frac t{D_k},x'+ke-\frac{ke}{D_k}\right).
\]
Then $X_k=\mathcal I_{a_k}(r,x')$, and the two points have the same tangential coordinates with $(X_k)_1<(Y_k)_1<0$. Hence the strict monotonicity of the Kelvin transform $v_k$ centered at $a_k$ gives $v_k(X_k)<v_k(Y_k)$. A direct calculation yields
\begin{equation}\label{eq6.16}
Q_k:=\mathcal I_{a_k}(Y_k) =\left(\frac{tD_k}{H_k}, x'+\frac{k(t^2-r^2)}{H_k}e\right) \longrightarrow(t,x').
\end{equation}
Moreover, $|X_k-a_k|^2=D_k^{-1}$ and $|Y_k-a_k|^2=H_kD_k^{-2}$, so the Kelvin factors give
\[
u(r,x')<\left(\frac{D_k}{H_k}\right)^{d/2}u(Q_k).
\]
Since $D_k/H_k\to1$, \eqref{eq6.16} and continuity at $(t,x')$ prove \eqref{eq6.15}.
\end{proof}

\begin{proof}[Proof of Theorem~\ref{Thm1.3}]
Lemma~\ref{Lem6.1} supplies the locally continuous representative and the local bounds needed above. Propositions~\ref{Prop6.1}-\ref{Prop6.3} apply to every regular center in the hyperplane. Lemma~\ref{Lem6.3} gives symmetry and weak transverse monotonicity of $u$. We next show $u\le u_\lambda$ in $\Sigma_\lambda\setminus\Gamma_\lambda$. If $x_1<\lambda<0$ and $y_1=2\lambda-x_1<0$, \eqref{eq6.15} gives $u(x_1,x')\le u(y_1,x')$. If $y_1=0$ and the reflected point is regular, the same inequality follows by continuity from the negative side. If $y_1>0$, symmetry and $x_1<-y_1=x_1-2\lambda<0$ give it again. Thus $u\le u_\lambda$ on $\Sigma_\lambda\setminus\Gamma_\lambda$ for every $\lambda<0$.

The reflection identity \eqref{eq4.1} now yields $V\le V_\lambda$. Hence $u_\lambda-u$ is an antisymmetric supersolution of $(-\Delta)^s(u_\lambda-u)\ge0$ on every ball $B\Subset\Sigma_\lambda\setminus\Gamma_\lambda$. Definition~\ref{Def1.2} and Lemma~\ref{Lem6.1} give $u_\lambda-u\in\mathcal D^s(B)$ and a locally bounded source. Lemma~\ref{Lem3.2} extends the inequality to every nonnegative $H_0^s(B)$ test, so Lemma~\ref{Lem5.1} applies with $c=0$. The identity $u\equiv u_\lambda$ would reflect a nonremovable singular point to a regular point and is excluded by local boundedness, exactly as in Proposition~\ref{Prop6.3}. Therefore $u<u_\lambda$ in $\Sigma_\lambda\setminus\Gamma_\lambda$ for every $\lambda<0$. Taking $\lambda=(r+t)/2$ proves strict increase for $r<t<0$; symmetry gives strict decrease on the positive half-space. This completes the proof.
\end{proof}

\section*{Acknowledgments}
This work is supported by the National Natural Science Foundation of China (12301145, 12561020, 12261107) and Yunnan Fundamental Research Projects (202401AU070123, 202601AT070048).
The authors acknowledge the use of ChatGPT (OpenAI) for exploratory discussions and language refinement during the preparation of this work. All mathematical statements, proofs, and references were independently checked by the authors, who take full responsibility for the content of the paper.

\medskip
{\bf Data availability:} Data sharing is not applicable to this article as no new data were created or analyzed in this study.

\medskip
{\bf Conflict of interest:} The authors declare that they have no conflict of interest.

\end{document}